\documentclass[11pt, reqno]{article}
\pdfoutput=1

\usepackage{amsmath, amsthm, amssymb}
\usepackage[normalem]{ulem}
\usepackage{blindtext}
\usepackage{bbm}
 \usepackage{ifpdf}
 \usepackage{float} 
\ifpdf
\usepackage[pdftex]{graphicx}
\else
\usepackage[dvips]{graphicx}
\uspackage{mathbbm}

\fi

\usepackage{fullpage}
\usepackage{enumerate,tcolorbox}
\usepackage{mathtools}

\newcommand{\makeheading}[1]%
        {\hspace*{-\marginparsep minus \marginparwidth}%
         \begin{minipage}[t]{\textwidth}%
                {\large \bfseries #1}\\[-0.15\baselineskip]%
                 \rule{\columnwidth}{1pt}%
         \end{minipage}}

\theoremstyle{plain}

\newtheorem{lemma}{Lemma}
\newtheorem{theorem}{Theorem}

\theoremstyle{definition}
\newtheorem{definition}{Definition}

\usepackage[pdftex,plainpages=false,hypertexnames=false,pdfpagelabels]{hyperref}

\def\semicolon{;}
\def\applytolist#1{
    \expandafter\def\csname multi#1\endcsname##1{
        \def\multiack{##1}\ifx\multiack\semicolon
            \def\next{\relax}
        \else
            \csname #1\endcsname{##1}
            \def\next{\csname multi#1\endcsname}
        \fi
        \next}
    \csname multi#1\endcsname}

\def\calc#1{\expandafter\def\csname c#1\endcsname{{\mathcal #1}}}
\applytolist{calc}QWERTYUIOPLKJHGFDSAZXCVBNM;
\def\bbc#1{\expandafter\def\csname bb#1\endcsname{{\mathbb #1}}}
\applytolist{bbc}QWERTYUIOPLKJHGFDSAZXCVBNM;
\def\bfc#1{\expandafter\def\csname bf#1\endcsname{{\mathbf #1}}}
\applytolist{bfc}QWERTYUIOPLKJHGFDSAZXCVBNM;
\def\sfc#1{\expandafter\def\csname s#1\endcsname{{\sf #1}}}

\def\fc#1{\expandafter\def\csname f#1\endcsname{{\mathfrak #1}}}
\applytolist{fc}QWERTYUIOPLKJHGFDSAZXCVBNM;

\usepackage{xcolor}

\usepackage[symbol]{footmisc}

\usepackage{tikz}
\usetikzlibrary{decorations.pathreplacing,angles,quotes}
\title{Limiting Speed and Fluctuations for the Boundary Modified Contact Process}
\author{Andrew Heeszel}

\begin{document}

\maketitle
  \begin{abstract}
 The boundary modified contact process models an epidemic spreading in one dimension with two infection parameters, $\lambda_i$ and $\lambda_e$. Starting from a finite infected set, each edge of $\mathbb{Z}$ transmits the infection at rate $\lambda_i$ except for the rightmost and leftmost edges incident to infected vertices, which transmit the infection at rate $\lambda_e$. We show a strong law of large numbers and central limit theorem for the location of the rightmost infected vertex when $\lambda_i = \lambda_c$ and $\lambda_e = \lambda_c + \varepsilon$. We also show stretched exponential tail bounds in the fluctuations of the rightmost infected vertex, the extinction time of the process on the event of non-survival, and the probability of survival given the size of the initial infected region. Our results extend to the boundary modified contact process whenever $\lambda_c \leq \lambda_i < \lambda_e$, and solves an open problem first proposed by Andjel and Rolla in \cite{andjel2023contact}.
  \end{abstract}
\section{Introduction}
In this paper we show a law of large numbers and central limit theorem for the location of the rightmost infected vertex of the one-dimensional boundary modified contact process, answering an open question first proposed by Andjel and Rolla in \cite{andjel2023contact}. The boundary modified contact process was first studied by Durrett and Schinazi in \cite{durrett2000boundary}, and is formed by assigning a separate infection rate to the leftmost and rightmost infected edges of the standard one-dimensional contact process. We will let $\{ \xi_t \}_{t \geq 0}$ denote the boundary modified contact process, which is a jump process taking values in the set $\Sigma = \{ 0,1\}^{\mathbb{Z}}$. We say a site $x\in \mathbb{Z}$ is infected at time $t$ if $\xi_t(x) = 1$, and susceptible if $\xi_t(x) = 0$. Let $| \xi_t |$ be the number of infected sites in $\xi_t$. We also define the leftmost and rightmost infected sites to be,
\begin{equation}
    \begin{aligned}
        \mathcal{R}(\xi_t) &= \sup \{ x \in \mathbb{Z} \ | \ \xi_t(x) = 1 \}, \\
        \mathcal{L}(\xi_t) &= \inf \{x \in \mathbb{Z} \ | \ \xi_t(x) = 1 \}.
    \end{aligned}
\end{equation}
Infected sites recover at rate $1$ and immediately become susceptible to reinfection. Any infected site $x \notin \{ \mathcal{L}(\xi_t), \mathcal{R}(\xi_t) \}$ infects each of its neighbors at rate $\lambda_i$. The infected site $\mathcal{R}(\xi_t)$ infects its neighbor $\mathcal{R}(\xi_t) +1$ with rate $\lambda_e$, and neighbor $\mathcal{R}(\xi_t) -1$ with rate $\lambda_i$. Similarly, site $\mathcal{L}(\xi_t)$ infects its neighbor $\mathcal{L}(\xi_t) - 1$ with rate $\lambda_e$, and neighbor $\mathcal{L}(\xi_t) + 1$ with rate $\lambda_i$. This model aligns with the standard contact process whenever $\lambda_i = \lambda_e=\lambda$. 

We say $\xi_t$ survives when for all $t > 0$ there exists $x \in \mathbb{Z}$ so that $\xi_t(x)$ is infected. The standard one dimensional contact process is known to have a phase transition in its infection rate at the critical infection rate $\lambda_c$. Survival is possible whenever the infection rate $\lambda > \lambda_c$, and impossible whenever $\lambda \leq \lambda_c$ as shown in \cite{bezuidenhout1990critical}.

Durrett and Schinazi in \cite{durrett2000boundary} first showed the boundary modified contact process can survive on $\mathbb{Z}$ whenever $\lambda_e > 1$ and $\lambda_i$ is sufficiently large, and cannot survive when $\lambda_e \leq 1$. The authors similarly show survival is possible when $\lambda_i > \lambda_c$ when $\lambda_e$ is sufficiently large, and that survival is not possible when $\lambda_i < \lambda_c$. The authors additionally show the survival probability is an increasing function in $\lambda_i$ and $\lambda_e$. Andjel and Rolla in \cite{andjel2023contact} show that the model can survive in the case when $\lambda_i = \lambda_c$ and $\lambda_e > \lambda_c$, and in the case when $\lambda_i > \lambda_c$ and $\lambda_e = \lambda_c$.

We define the shift operator $\Psi: \Sigma \rightarrow \Sigma \cup \{ \dagger \}$ so that for any configuration $\eta \in \Sigma$,
\begin{equation}
    \Psi \eta(x) = \begin{cases} \eta \left( x - \mathcal{R}(\eta) \right) & \text{if }\eta \neq \emptyset \text{ and }  \mathcal{R}(\eta) \in \mathbb{Z} \\
    \dagger & \text{else} \end{cases}.
\end{equation}
Terra \cite{terra2024dynamic} shows that as $t \rightarrow \infty$ the shifted process $\Psi \xi_t$ converges weakly to an invariant measure $\tilde{\mu}$, whenever the initial state $\eta_0$ belongs to the set of states with infinitely many infections to the left of the origin and finitely many to the right,
\begin{equation}
    \Sigma^{\ominus} = \left\{ \eta \in \Sigma \ \mid \ \sum_{x > 0} \eta(x) < \infty, \ \sum_{x < 0} \eta(x) = \infty \right\}.
\end{equation}
Terra additionally shows that the invariant measure of the process as seen from the right edge, $\tilde{\mu}$, is supported on $\Sigma^{\ominus}$.

Let $\theta(\lambda_i, \lambda_e)$ be the survival probability of $\xi_t$ with infection rates $(\lambda_i, \lambda_e)$ and $\xi_0$ having a single infection at the origin. Terra  \cite{terra2024dynamic} also shows that when $\lambda_i > \lambda_c$, then $\xi_t$ cannot survive when $\lambda_e$ equals the critical boundary infection rate $\lambda_e^*(\lambda_i) = \inf \{ \lambda > 0 \mid \theta(\lambda_i, \lambda) > 0 \}$.

We study the regime when $\lambda_i = \lambda_c$, and $\lambda_e = \lambda_c + \varepsilon$ for $\varepsilon > 0$. This model is of interest since many of the standard tools in studying the supercritical contact process such as attractiveness, sub-additivity, and a direct comparison to supercritical oriented percolation are no longer applicable. We will often but not exclusively study $\xi_t$ with $\xi_0 \in \Sigma^{\ominus}$. In this regime the boundary modified contact process aligns with the right edge modified contact process $\{ \eta_t \}_{t \geq 0 }$, which only assigns infection rate $\lambda_e$ to the edge $(\mathcal{R}(\eta_t), \mathcal{R}(\eta_t) + 1)$, and rate $\lambda_i$ to all other regions. In this paper we will show a strong law of large numbers and central limit theorem for both $\mathcal{R}(\eta_t)$ and $\mathcal{R}(\xi_t)$. We will also provide stretched exponential estimates for the deviations of $\mathcal{R}(\eta_t)$ from its mean, and stretched exponential tails to the extinction time $\tau^{\emptyset}$ of either $\xi_t$ or $\eta_t$ on the event of non-survival, along with the probability of extinction given the initial outbreak size. These stretched exponential bounds are analogous to exponential bounds for the supercritical contact process obtained as a consequence of sub-additivity and the comparison to oriented percolation first shown in \cite{durrett1983supercritical}. 

We will also use the following lemma shown by Andjel and Rolla given in the proof of Theorem 5 of \cite{andjel2023contact}.
\begin{lemma}[Andjel, Rolla]
    \label{halflinesurvival}
    Let $\xi_t$ be either the boundary modified or right edge modified contact process with infection rates $\lambda_i = \lambda_c$ and $\lambda_e = \lambda_c + \varepsilon$ for some $\varepsilon > 0$. Then when $|\xi_0| < \infty$,
    \begin{equation}
        \mathbb{P}\left( \xi_t \text{ survives and } \liminf_t \frac{\mathcal{R}(\xi_t)}{t} > 0 \right) > 0.
    \end{equation}
\end{lemma}

We also leverage a box crossing property for critical two-dimensional oriented percolation. This was first shown by Duminil-Copin, Tassion, and Teixeira in \cite{duminil2018box}. The authors show as $n \rightarrow \infty$, there exists a scaling $w(n) \rightarrow \infty$ so that the probability that a $w(n) \times n$ box is crossed horizontally and vertically for critical two-dimensional oriented percolation is positive and independent of $n$. We say that the box $B = [0, n_1] \times [0, n_2]$ is crossed vertically when there exists an open path from two sites $x$ and $y \in [0, n_1]$ between times $0$ and $n_2$ contained in $B$. The box $B$ is crossed horizontally when there exist times $0\le t_1<t_2\le n_2$ so that there is an open path from $0$ to $n_1$ contained in $B$ between times $t_1$ and $t_2$. We will formally define open paths for the contact process in section 4. The box crossing property also holds for the critical contact process, giving us the following lemma,
\begin{lemma}[Duminil-Copin, Tassion, Teixeira]
    \label{boxcrossingdefinition}
    Let $\zeta_t$ be the critical contact process. There exists a scaling $w(n) \rightarrow \infty$ such that for any $n \in \mathbb{N}$ there is a probability of at least $p > 0$ independent of $n$ that a box of dimension $w(n) \times n$ is crossed both horizontally and vertically by $\zeta_t$. Furthermore there exist universal constants $c > 0$ and $\delta \in (0,1)$ so that for all $n \in \mathbb{N}$,
    \begin{equation}
        w(n) \leq c n^{1-\delta}.
    \end{equation}
\end{lemma}
The function $w(t)$ scales on the same order of magnitude as the right edge process $\mathcal{R}(\zeta_t)$ of the critical contact process when $(-\infty, 0]$ is initially infected. Using an identical argument as in (4.3) of \cite{duminil2018box} we have the following lemma.
\begin{lemma}[Duminil-Copin, Tassion, Teixeira]   
    \label{boxcrossrt}
    Let $\zeta_t$ denote the critical contact process with initial configuration $\zeta_0 = (-\infty, 0]$. Then there exist universal constants $c_1$, $c_2 > 0$ so that for any $y > 0$,
    \begin{equation}
        \mathbb{P}\left( \sup_{0 \leq s \leq t} \left| \mathcal{R}(\zeta_s) \right| > y t^{1-\delta} \right) \leq c_1 \exp\left( - c_2 y \right).
    \end{equation}
\end{lemma}
\section{Main Results}

We answer the conjecture in \cite{andjel2023contact} by providing both a law of large numbers and central limit theorem for the right edge process of the boundary modified contact process when $\lambda_e > \lambda_i \geq \lambda_e$. Our results are in the setting where the initial infected region is in the half space $\Sigma^{\ominus}$ and is applicable to both $\xi_t$ and $\eta_t$. We show that the asymptotic speed of the right edge converges to the constant,
\begin{equation}
    \alpha = \mathbb{E}\left( \mathcal{R}(\tilde{\eta}_1)\right),
\end{equation}
where $\tilde{\eta}_t$ is the right edge modified contact process with $\tilde{\eta}_0$ sampled via $\tilde{\mu}$. We now state the law of large numbers and central limit theorem. 
\begin{theorem}
    \label{mainthmlln}
    Let $\{ \eta_t \}_{t \geq 0}$ be the right edge modified contact process with infection rates $\lambda_i = \lambda_c$, $\lambda_e = \lambda_c + \varepsilon$ for $\varepsilon >0$, and $\eta_0 \in \Sigma^{\ominus}$. Then almost surely as $t \rightarrow \infty$,
    \begin{equation}
        \frac{\mathcal{R}(\eta_t)}{t} \rightarrow \alpha.
    \end{equation}
    Where the constant $\alpha = \mathbb{E}\left( \mathcal{R}(\tilde{\eta}_1) \right)$ additionally satisfies $\alpha \geq \varepsilon$.
\end{theorem}
\begin{theorem}
    \label{mainthmbrownianmotion}
    Let $\{ \eta_t \}_{t \geq 0}$ be the right edge modified contact process with infection rates $\lambda_i = \lambda_c$, $\lambda_e = \lambda_c + \varepsilon$ for $\varepsilon >0$, and $\eta_0 \in \Sigma^{\ominus}$. Then as $n \rightarrow \infty$
    \begin{equation}
        \left\{ \frac{1}{\sqrt{n}} \left( \mathcal{R}(\eta_{nt}) - \alpha n t  \right) \right\}_{t \geq 0} \Rightarrow \{ W_t  \}_{t \geq 0}.
    \end{equation}
    Where $\{ W_t \}_{t \geq 0 }$ is Brownian motion with a drift coefficient $\sigma^2 > 0$ independent of $\eta_0$, and $\Rightarrow$ denotes convergence in distribution.
\end{theorem}
When $\lambda_e > \lambda_i$ the boundary modified contact process is no longer attractive and the right edge $\mathcal{R}(\eta_t)$ is no longer subadditive. As a consequence we are no longer able to apply the subadditive ergotic theorem in a similar format as the standard contact process. We instead provide a stretched exponential tail bound to the extinction time $\tau^\emptyset$ on the event of non-survival. This coupled with the attractiveness property at the right edge allows us to deduce that the increments of $\mathcal{R}(\eta_t)$ are strongly mixing. To show this we use the box crossing property for critical oriented percolation and Lemma \ref{halflinesurvival} in order to form an event sufficient for survival that occurs with a high probability. As a result, we attain theorem \ref{mainthmsurvivaltime}.
\begin{theorem}
    \label{mainthmsurvivaltime}
    Let $\{ \xi_t \}_{t \geq 0}$ be either the boundary modified or right edge modified contact process with infection rates $\lambda_i = \lambda_c$, $\lambda_e = \lambda_c + \varepsilon$, with $\varepsilon >0$, and $|\xi_0| = n< \infty$. Let $\tau^\emptyset$ be the first hitting time of $\xi_t$ to the all susceptible configuration $\emptyset$. Then there exist universal constants $c$, $c'$, and $a>0$ independent of $n$ so that for all $t > 0$,
    \begin{equation}
    \mathbb{P}\left(t < \tau^\emptyset < \infty \right) \leq c \exp \left( -c' t^a\right),
    \end{equation}
    where the constant $a$ is independent of $\varepsilon$.
\end{theorem}
As a corollary of Theorem \ref{mainthmsurvivaltime} and the box crossing property, we then get stretched exponential bounds on the probability of extinction based on the size of the initial infected set. \\
\begin{theorem}
    \label{mainthmsurvivalsize}
     Let $\{ \xi_t \}_{t \geq 0}$ be either the boundary modified or right edge modified contact process with infection rates $\lambda_i = \lambda_c$, $\lambda_e = \lambda_c + \varepsilon$, with $\varepsilon >0$, and $|\xi_0| = n< \infty$. Let $\tau^\emptyset$ be the first hitting time of $\xi_t$ to the all susceptible configuration $\emptyset$. Then there exists constants $c, c'$, and $a > 0$ so that,
     \begin{equation}
         \mathbb{P}_{\xi_0} \left( \tau^\emptyset < \infty \right) \leq c \exp \left( -c' n^a \right),
     \end{equation}
     where the constant $a$ is independent of $\varepsilon$.
\end{theorem}
We use the results of Theorem \ref{mainthmsurvivaltime} and an attractiveness property at the right edge first used in \cite{terra2024dynamic} to show the increments of $\mathcal{R}(\eta_t)$ are strongly mixing at a stretched exponential rate, and gain control on the tails of the increments of $\mathcal{R}(\eta_t)$ via a similar argument to Galves and Presutti in \cite{galves1987edge}. This allows us to show stretched exponential deviations for $\mathcal{R}(\tilde{\eta}_t)$ from its mean when $\tilde{\eta}_0$ is sampled via $\tilde{\mu}$, as well as the strong law of large numbers and central limit theorem for $\eta_t$ when $\eta_0 \in \Sigma^{\ominus}$. We now state our large deviation inequality for $\mathcal{R}(\tilde{\eta}_t)$.
\begin{theorem}
    \label{mainthmrttail}
        Let $\tilde{\eta}_t$ be the right edge modified contact process with infection rates $\lambda_i = \lambda_c$, $\lambda_e = \lambda_c + \varepsilon$ for $\varepsilon > 0$, and that $\tilde{\eta}_0$ is sampled from $\tilde{\mu}$. Then for any $\gamma \in [0, \frac{1}{3})$ and $b > 0$ there exist universal constants $c, c'$ and $a>0$ so that for all $t > 0$,
    \begin{equation}
        \mathbb{P} \left( \left| \mathcal{R}(\tilde{ \eta}_t) - \alpha t\ \right| > b t^{1-\gamma} \right) \leq c \exp\left( -c' t^a \right),
    \end{equation}
    where the constant $a$ only depends on $\gamma$ and is independent of $\varepsilon$.
\end{theorem}
The critical contact process is known to die out as shown in \cite{bezuidenhout1990critical}. Moreover in one dimension the survival function of the critical contact process, $s(t)$, has been bounded below by a polynomial in \cite{durrett1988lecture} and above by a polynomial in \cite{duminil2018box}. Both bounds assume the contact process has a single infection at the origin. Our results now show a sharp transition when adding the boost of $\varepsilon$ to the left and right edges $\lambda_e$. By Theorem \ref{mainthmsurvivaltime} we see the survival function of the process with the origin initially infected now decays at no slower than a stretched exponential rate on the event of extinction. For the supercritical contact process it is shown in \cite{durrett1983supercritical} that the survival function of the process with a single infection at the origin decays at an exponential rate. Determining if the survival function of the edge modified contact process has heavier than exponential tails on the event of extinction when $\lambda_i = \lambda_c$ and $\lambda_e > \lambda_c$ remains an open question. We conjecture Theorem \ref{mainthmrttail} can be used to show a law of large numbers for the outbreak size $|\xi_t|$, since the theorem gives stretched exponential bounds for deviations of $\mathcal{R}(\xi_t)$ being of order $t^{1-\gamma}$ for $\gamma \in [0, \frac{1}{3})$. On the other hand, we expect the outbreak size $|\xi_t|$ to scale of order no less that $t s(t) \gtrsim t^{\frac{4}{5}}$ based on Theorem 1.1 of \cite{duminil2018box}. If we set $\gamma = 0.3$ in Theorem \ref{mainthmrttail}, then we have stretched exponential bounds for fluctuations in the right edge of a lower order of magnitude than the conjectured mean outbreak size $|\xi_t|$.

\section{Notation and a Note on Constants}

We will denote $\mathbb{N}$ as the natural numbers $\{0,1,2,\ldots \} $ and $\mathbb{Z}_{0+}$ as the non-negative integers. All numbered constants such as $c_1$, $c_2$, defined throughout the paper will typically be referred to in later parts and be held constant line by line. On the other hand the constants $c$ and $c'$ are assumed to be interchanging line by line. If specified constants may be independent of other parameters varied throughout the arguments of this paper, such as the time $t$. We use standard asymptotic notation,
\begin{equation}
    \begin{aligned}
        f &= \mathcal{O}(g) \text{ if } \limsup \frac{f}{g} < \infty \\
        f &= o(g) \text{ if } \lim \frac{f}{g} = 0
    \end{aligned}
\end{equation}
We will commonly consider the right edge modified contact process $\eta_t$ supported on the set $\Sigma^{\ominus}$; this probability model is equivalent to the boundary modified contact process $\xi_t$ supported on $\Sigma^{\ominus}$. For any time $t \in \mathbb{R}$ we define the time $t^{-}$ as,
\begin{equation}
    \xi_{t^-}(x) = \lim_{s \rightarrow 0^-} \xi_{t+s}(x).
\end{equation}
A site $x$ being infected at time $t^-$ also corresponds to the event $\xi_{t^-}(x) = 1$.

\section{Sitewise Construction \& Attractiveness at the Right Edge}
We will now generate a graphical construction for the boundary modified contact process $\xi_t$ and right edge modified contact process $\eta_t$. Let $\tilde{E}$ be the directed edge set of $\mathbb{Z}$. For each directed edge $(x,y) \in \tilde{E}$, let $N_{x,y}(t)$ be a rate $\lambda_c$ Poisson process. For each site $x \in \mathbb{Z}$, let $N_x(t)$ be a rate $1$ Poisson process. Let $N_{e1}(t)$ and $N_{e2}(t)$ each be rate $\varepsilon$ Poisson processes. Suppose that each of previous counting processes are independent. We evolve $\xi_t$ and $\eta_t$ by the following rules,
\begin{enumerate}
    \item If a site $x \in \mathbb{Z}$ is infected at time $t^-$  for either $\xi_{t^-}$ or $\eta_{t^-}$, and the recovery clock $N_x(t)$ rings at time $t$, then both $\xi_t(x)$ and $\eta_t(x)$ will jump to being susceptible at time $t$. 
    \item If at time $t^-$ a site $x \in \mathbb{Z}$ is infected and $x+1$ is susceptible for either $\xi_{t^-}$ or $\eta_{t^-}$, and the infection clock $N_{x,x+1}(t)$ rings at time $t$, then both $\xi_t(x+1)$ and $\eta_t(x+1)$ will jump to being infected at time $t$.
    \item If at time $t^-$ a site $x \in \mathbb{Z}$ is infected and $x-1$ is susceptible for either $\xi_{t^-}$ or $\eta_{t^-}$, and the infection clock $N_{x,x-1}(t)$ rings at time $t$, then both $\xi_t(x-1)$ and $\eta_t(x-1)$ will jump to being infected at time $t$.
    \item Whenever the clock $N_{e1}(t)$ rings, $\xi_t$ will send an infection from the random location $\mathcal{L}(\xi_t)$ to location $\mathcal{L}(\xi_t) - 1$, resulting in $\xi_t\left( \mathcal{L}(\xi_t) -1  \right)$ becoming infected. There is no change if $\xi_t = \emptyset$.
    \item Whenever the clock $N_{e2}(t)$ rings, both $\xi_t$ and $\eta_t$ will send an infection from the random locations $\mathcal{R}(\xi_t)$, and $\mathcal{R}(\eta_t)$ to locations $\mathcal{R}(\xi_t) + 1$, and $\mathcal{R}(\eta_t) + 1$, respectively. There is no change if $\xi_t = \emptyset$ or $\eta_t = \emptyset$, respectively.
\end{enumerate}
When the left or right edge is equal to $\pm \infty$, then we will not attribute any edge boosts from the clocks $N_{e1}(t)$ or $N_{e2}(t)$, respectively. We will define our outcome space,
\begin{equation}
    \Omega = \Omega_{e1} \times \Omega_{e2} \times \left( \prod_{x \in \mathbb{Z}}  \Omega_x \times \Omega_{x,x+1} \times \Omega_{x,x-1} \right),
\end{equation}
where $\Omega_{e1}, \  \Omega_{e2}, \ \Omega_{x},  \ \Omega_{x,x+1}, \Omega_{x,x-1}$ are each the set of c\'adlag functions on $\mathbb{Z}_{0+}$ for $x \in \mathbb{Z}$. We define our sigma-field $\mathcal{F}$ to be the minimal sigma field generated by $\{N_{e1}(t)\}_{t \geq 0}$, $\{N_{e2}(t)\}_{t \geq 0}$, $\{N_{x}(t)\}_{t \geq 0}$, $\{N_{x,x+1}(t)\}_{t \geq 0}$, and $\{N_{x,x-1}(t)\}_{t \geq 0}$ for each $x \in \mathbb{Z}$. We let our probability measure $\mathbb{P}$ be the product measure between the laws of $\{N_{e1}(t)\}_{t \geq 0}$, $\{N_{e2}(t)\}_{t \geq 0}$, $\{N_{x}(t)\}_{t \geq 0}$, $\{N_{x,x+1}(t)\}_{t \geq 0}$, and $\{N_{x,x-1}(t)\}_{t \geq 0}$ for each $x \in \mathbb{Z}$. We will let $ \{ \mathcal{F}_t \}_{t \geq 0}$ be the natural filtration of the sitewise construction of the boundary modified contact process up until time $t$.

We say that a \textbf{$\lambda_i$-open path} exists between sites $x$ and $y$ between times $0$ and $t$ if an open path exists under the standard graphical construction of the critical contact process (infection paths omitting the edge-boost clocks $N_{e1}(t)$ and $N_{e2}(t)$). Namely there exists a $\lambda_i$ open path between $x$ and $y$ between times $0$ and $t$, if there exists a sequence,
\begin{equation}
    \begin{aligned}
        (x,0), (x_1, t_1), \ldots, (x_{n-1}, t_{n-1}), (y, t_n)
    \end{aligned}
\end{equation}
So that for $0 \leq i \leq n-1$ we have that $\xi_s(x_i) = 1$ in $[t_i, t_{i+1})$, that $|x_i - x_{i+1}| = 1$, and at time $t_{i+1}$ we observe $N_{x_i, x_{i+1}}(s)$ increment. Here we assume that $x_0 = x$, $t_0 = 0$, $x_n = y$, and $\xi_0(x) = 1$. We say that a \textbf{$\lambda_e$-open path exists for $\xi_t$} between sites $x$ and $y$, times $0$ and $t$ and initial infected set $A$ if there exists an open path under the graphical construction above when additionally including infections from the edge boosts $N_{e1}(t)$ and $N_{e2}(t)$ and setting $\xi_0 = A$. We say that a \textbf{$\lambda_e$-open path exists for $\eta_t$} between sites $x$ and $y$, times $0$ and $t$ and initial infected set $A$ if there exists an open path under the graphical construction above when only using the right edge boost $N_{e2}(t)$ and setting $\eta_0 = A$.

We will now give an attractiveness property first used by Terra in \cite{terra2024dynamic} shared by both $\xi_t$ and $\eta_t$ about the right edge. We define the auxiliary process,
\begin{definition}
    \label{rightedgeauxillary}
    Let $\{ \xi_t \}_{t \geq 0}$ be either a copy of the boundary modified or right edge modified contact process. We define the process $\{ \eta^t_s \}_{s \geq 0 } $ as a copy of the right edge modified contact process so that $\eta^t_0 = \{0 \}$. For any site $x \in \mathbb{Z}$ and $s \geq 0$, $\eta^t_s$ will use infection clocks from $N_{x + \mathcal{R}(\xi_t), x + \mathcal{R}(\xi_t)-1}(t + s)$ and $N_{x + \mathcal{R}(\xi_t), x + \mathcal{R}(\xi_t)+1}(t + s)$ for infections occurring to the left and right respectively. Site $x$ will also recover at time $s$ if the clock $N_{x + \mathcal{R}(\xi_t)}(t+s)$ increments. $\eta^t_s$ will have a right edge boost at time $s$ whenever $N_{e2}(t+s)$ increments. Our notation will set $\mathcal{R}(\xi_t) = 0$ whenever $\xi_t = \emptyset$.
\end{definition}

We see via this construction that $\eta^t_s$ will have the law of the right edge modified contact process with the origin initially infected. If $\tau^{\emptyset}$ is the first hitting time of $\eta^t_s$ to the all susceptible state, then for all $s \in [0, \tau^\emptyset)$,
\begin{equation}
    \mathcal{R}(\xi_{t+s}) = \mathcal{R}(\xi_t) + \mathcal{R}(\eta^t_s).
\end{equation}
Moreover, we also have that for all $s \in [0, \tau^\emptyset)$ that the region,
\begin{equation}
    \left[ \mathcal{R}(\xi_t) - \mathcal{L}(\eta^t_s), \mathcal{R}(\xi_t) + \mathcal{R}(\eta^t_s) \right],
\end{equation}
is coupled between $\xi_{t+s}$ and $\eta^t_s$. We will refer to this coupled region as the \textbf{attractiveness property at the right edge} for $\xi_t$ and $\eta_t$. We will commonly use this property throughout this paper to couple the shifted process $\Psi \xi_{t+s}$ with $\Psi \xi'_s$, where $\xi'_s$ is a copy of the boundary modified contact process beginning with initial configuration $(-\infty, 0]$, or sampled from the invariant measure $\tilde{\mu}$.

\section{Bounding Survival Function of $\eta(t)$ on Extinction}

 We will now bound the survival function of $\eta'(t)$, the right edge modified contact process with the origin initially infected. We will then use this bound to prove Theorems \ref{mainthmsurvivaltime} and \ref{mainthmsurvivalsize}. We will now show,
\begin{lemma}
    \label{survivefunctnbound}    
    Let $\eta'_t$ be the right edge modified contact process with infection rates $\lambda_i = \lambda_c$, $\lambda_e = \lambda_c + \varepsilon$ for $\varepsilon > 0$ and the origin initially infected. Let $\tau^\emptyset$ be the first hitting time of $\eta'_t$ to the all susceptible configuration. Then there exists universal constants $c_9, c_{10} > 0$ so that,
    \begin{equation}
        \mathbb{P}\left( t < \tau^{\emptyset} < \infty \right) \leq c_9 \exp\left( - c_{10} t^{\frac{\delta}{4}} \right).
    \end{equation}
    Where $\delta \in (0,1)$ is defined as in Lemma \ref{boxcrossingdefinition}.
\end{lemma}
\subsection{Bounding forward increments of $\mathcal{R}(\eta'_t)$ on event of survival by time $t$}
Our first step in showing Lemma \ref{survivefunctnbound} is to show that survival by time $t > 0$ implies that $\mathcal{R}(\eta'_s)$ has exceeded $t^{1-\frac{\delta}{2}}$ with a high probability for some $0 \leq s \leq t$. To do this we will first define the constant $\beta > 0$ so that,
\begin{equation}
    \label{definebeta}
    \mathbb{P}\left( \eta'_t \text{ survives and } \liminf_{t \rightarrow \infty} \frac{\mathcal{R}(\eta'_t)}{t} > \beta \right) = q > 0.
\end{equation}
The existence of $\beta > 0$ is given by Lemma  \ref{halflinesurvival}. We will now show the following.
\begin{lemma}
    \label{survivalsomespeed}
    Let $\eta'_t$ denote the right edge modified contact process with the origin initially infected and infection rates $\lambda_i = \lambda_c$, and $\lambda_e = \lambda_c + \varepsilon$. Let $A_t$ be the event of survival by time $t$, and fix $y > 0$. Then there exists universal constants $c_3, c_4 > 0$ independent of $t$ so that,
    \begin{equation}
        \label{boxcrossworstcase}
        \mathbb{P}\left( A_t \cap \left\{  \sup_{0 \leq s \leq t} \mathcal{R}(\eta'_s) \leq y t^{1-\frac{\delta}{2}} \right\}\right) \leq c_3 \exp \left( - c_4 t^{\frac{\delta}{2}} \right).
    \end{equation}
    \begin{proof}
        By the attractiveness property at the right edge we can couple $\mathcal{R}(\eta'_s)$ with $\mathcal{R}(\bar{\eta}_s)$ on $A_t$ for $0 \leq s \leq t$, where $\bar{\eta}_s$ is the right edge modified contact process with $(-\infty, 0]$ infected for $0 \leq s \leq t$. We also know that $\mathcal{R}(\bar{\eta}_s)$ will dominate $\mathcal{R}(\zeta_s)$ for all $s \geq 0$ where $\zeta_s$ is the critical contact process with $(-\infty, 0]$ initially infected. Thus since $\eta'(s)$ dominates the critical contact process and using Lemma \ref{boxcrossrt} we have,
        \begin{equation}
            \label{setupsomespeed}
            \mathbb{P}\left( A_t \cap \inf_{0 \leq s \leq t} \mathcal{R}(\eta's) \leq - y  t^{1-\frac{\delta}{2}}  \right) \leq c_1 \exp\left( -c_2 y t^{\frac{\delta}{2}} \right).
        \end{equation}
        By Lemma \ref{halflinesurvival} and (\ref{definebeta}) we also have that,
        \begin{equation}
        \label{speedpossible}
            \liminf_{t \rightarrow \infty} \ \mathbb{P}\left( \frac{\mathcal{R}(\eta'_{t})}{t} > \frac{\beta}{2} \right) = q' > 0.
        \end{equation}
For $n \in \mathbb{Z} \cap [ 0, \dfrac{\beta t^{\frac{\delta}{2}}}{4y} -1 ] $ we let $t_n = \dfrac{4yt^{1-\frac{\delta}{2}}}{\beta}$, $t' = \left\lfloor \dfrac{\beta t^{\frac{\delta}{2}}}{4y} -1 \right\rfloor $ and define the event,
\begin{equation}
    B_n = \left\{ \mathcal{R}(\eta^{t_n}_{t'}) \geq 2yt^{1-\frac{\delta}{2}} \right\}.
\end{equation}

By the Markov property we can conclude that the collection of events $\{ B_i \}_{i = 0}^{t'} $ are independent. For $0 \leq n \leq t'$ we also have by (\ref{speedpossible}) that,
\begin{equation}
    \mathbb{P}\left( B_n \right) \geq q'.
\end{equation}
And thus,
\begin{equation}
    \label{somebi}
    \mathbb{P}\left( \cap_{n=0}^{t'} B_n^c \right) \leq \left( 1 - q' \right)^{ \left\lfloor \dfrac{\beta t^{\frac{\delta}{2}}}{4y} -1 \right\rfloor }.
\end{equation}
We now intersect on the event,
\begin{equation}
    A_t \cap \{ \inf_{0 \leq s \leq t} \mathcal{R}(\eta'(s) \geq -yt^{1-\frac{\delta}{2}} \} \cap \{ \cup_{n=0}^{t'} B_n \},
\end{equation}
and assume the event $B_n$ occurred for $0 \leq n \leq t'$. Using the attractiveness property at the right edge we can write setting $t'' = \frac{4y (n+1)t^{1-\frac{\delta}{2}}}{\beta}$,
\begin{equation}    
    \label{somespeedattractiveness}
    \mathcal{R}(\eta'_{t''}) = \mathcal{R}(\eta^{t_n}_{t'' - t_n}) + \mathcal{R}(\eta'_{t_n}) \geq y t^{1-\frac{\delta}{2}}.
\end{equation}
Thus by (\ref{setupsomespeed}), (\ref{somebi}), and (\ref{somespeedattractiveness}), we have completed the lemma.
    \end{proof}
\end{lemma}
\subsection{Forming a High Probability Survival Event}
We will start this section by showing a lemma related to the box crossing property for critical oriented percolation in two dimensions.
\begin{lemma}
    \label{boxcrossingproportion}
    There exists constants $c_5, c_6 > 0$ such that for all $n \geq 1$ the probability that a box of dimension $n \times n^{\frac{2-\delta}{2(1-\delta)}}$ is not crossed vertically under the critical contact process is less than or equal to,
    \begin{equation}
        c_5 \exp\left( - c_6 n^{\frac{\delta}{2}} \right).
    \end{equation}
    \begin{proof}
    We write our box $B = \left[ 0, n \right] \times \left[ 0, n^{\frac{2-\delta}{2(1-\delta)}} \right]$, and note that the function $w(n)$ in Lemma \ref{boxcrossingdefinition} satisfies $w(n) \leq cn^{1-\delta}$. For $1 \leq i \leq  \left\lfloor c^{-1} n^{\frac{\delta}{2}} \right\rfloor$ we form disjoint subsets,
    \begin{equation}
        B_i = \left[ (i-1) cn^{1-\frac{\delta}{2}}, i cn^{1-\frac{\delta}{2}} \right] \times \left[0, n^{\frac{2-\delta}{2(1-\delta)}} \right].
    \end{equation}
    Using Lemma \ref{boxcrossingdefinition}, there exists a constant $p >0$ independent of $n$ so that for $1 \leq i \leq  \left\lfloor c^{-1} n^{\frac{\delta}{2}} \right\rfloor$,
    \begin{equation}
        \mathbb{P}\left( B_i \text{ is crossed vertically} \right) \geq p.
    \end{equation}
    Using that our collection $\left\{ B_i \right\}_{i=1}^{\left\lceil c^{-1}n^{\frac{\delta}{2}} \right\rceil }$ is disjoint, we can write,
    \begin{equation}
    \begin{aligned}
        & \mathbb{P}\left( B \text{ is crossed vertically} \right) \\
        &\geq \mathbb{P}\left( \bigcup_{i=1}^{\left\lfloor c^{-1} n^{\frac{\delta}{2}} \right\rfloor} \left\{ B_i \text{ is crossed vertically} \right\} \right) \\
        &\geq 1- (1-p)^{\left\lfloor c^{-1} n^{\frac{\delta}{2}} \right\rfloor},
    \end{aligned}
    \end{equation}
    which completes the lemma.
    \end{proof}
\end{lemma}
For $k \in \mathbb{N}$ let $K$ be the first hitting time of $\mathcal{R}(\eta'_t)$ to $k$, and let, 
\begin{equation}
    D_k = \left\{ K \leq \left( \frac{k}{2} \right)^{\frac{2-\delta}{2(1-\delta)}} \right\}.
\end{equation}
Letting $\tau^\emptyset$ be the first hitting time of $\eta'_t$ to the all susceptible state, we will now show the following lemma,
\begin{lemma}
    \label{dkrare}
    Let $\eta'(t)$ be the right edge modified contact process with infection rates $\lambda_i = \lambda_c$, and $\lambda_e = \lambda_c + \varepsilon$ for $\varepsilon > 0$ and the origin initially infected. Then there exists universal constants $c_7, c_8 > 0$ independent of $k$ so that for the event $D_k$,
    \begin{equation}
        \mathbb{P}\left( \{\tau^\emptyset < \infty \} \cap D_k \right) \leq c_7 \exp\left(- c_8 k^{\frac{\delta}{2}} \right).
    \end{equation}
\begin{proof}
    We will define the event $G_k$ as the event that the box $\left[0, k \right] \times \left[0, k^{\frac{2-\delta}{2(1-\delta)}}\right]$ is crossed vertically under the critical contact process (ignoring the infection clock $N_{2e}(t)$ entirely). Using Lemma \ref{boxcrossingproportion} there exists constants $c_5$, $c_6 > 0$ so that for any $k \geq 0$,
    \begin{equation}
        \label{gkthat}
        \mathbb{P}\left( G_k \right) \geq 1 - c_5 \exp\left(-c_6 k^{\frac{\delta}{2}} \right).
    \end{equation}
    We will now define the event $F_k$ as follows;
    \begin{enumerate}
        \item For $n \in \left[ 0, \dfrac{\beta k^{\frac{\delta}{2(1-\delta)}}}{4} -1 \right] \cap \mathbb{Z}$ we let $t_{n,k} = \left( \frac{k}{2} \right)^{\frac{2-\delta}{2(1-\delta)}} + \frac{4nk}{\beta}$, and $b_k = \left\lfloor \frac{\beta k^{\frac{\delta}{2(1-\delta)}} }{4} -1\right\rfloor$
        \item We say the event $F_{k,n}$ occurs if $ \mathcal{R}(\eta^{t_{n,k}}_{\frac{4k}{\beta}}) \geq 2k$
        \item We then set $F_k = \cup_{n=0}^{b_k} F_{k,n}$
    \end{enumerate}
    We will now show that,
    \begin{equation}
    \label{impliessurvival}
        D_k  \cap \left( \cap_{j=0}^\infty \left( F_{k2^j} \cap G_{k2^j} \right) \right) \subseteq \{ \tau^\emptyset < \infty \}.
    \end{equation}       
    Figure 1 provides a visual for the description provided below. We first see that the event $D_k$ guarantees that $\mathcal{R}(\eta'_s)$ reaches $k$ at or prior to time $\left( \frac{k}{2} \right)^{\frac{2-\delta}{2(1-\delta)}}$. The event $G_k$ then gives that the box $B_k =  \left[ 0,k \right] \times \left[ 0, k^{\frac{2-\delta}{2(1-\delta)}} \right]$ is crossed vertically by the critical contact process. This means there exists a $\lambda_i$-open path from two sites $x, y \in [0,k]$ between times $0$ and $k^{\frac{2-\delta}{2(1-\delta)}}$ that stays within the box $B_k$. This $\lambda_i$-open path between $x$ and $y$ must then intersect the $\lambda_e$-open path between $0$ and $k$. This lets us then conclude that, 
    \begin{equation}
    \mathcal{R}(\eta'_t) \geq 0 \text{ for all } t \in \left[ \left( \frac{k}{2} \right)^{\frac{2-\delta}{2(1-\delta)}},  k ^{\frac{2-\delta}{2(1-\delta)}}\right].
    \end{equation}
    Let us now suppose that for $j \geq 0$ that,
    \begin{equation}
        \label{positiverighteedge}
        \mathcal{R}(\eta'_t) \geq 0, \text{ for all } t \in \left[  \left( k 2^{j-1} \right)^{\frac{2-\delta}{2(1-\delta)}},  \left( k 2^{j} \right)^{\frac{2-\delta}{2(1-\delta)}}\right].
    \end{equation}
    On the event that $F_{k2^j ,n}$ holds for some $n \in \mathbb{Z} \cap \left[0, \frac{\beta (k2^{j})^{\frac{\delta}{2(1-\delta)}}}{4}-1 \right]$ we can then apply the attractiveness property at the right edge and (\ref{positiverighteedge}) to conclude,
    \begin{equation}
        \mathcal{R}(\eta'_{t_{n+1,k2^j}}) =  \mathcal{R}(\eta'_{t_{n,k2^j}}) + \mathcal{R}(\eta^{t_{n,k2^j}}_{\frac{k}{\beta} 2^{j+2}}) \geq k 2^{j+1}.
    \end{equation}
    And thus $\mathcal{R}(\eta'_{t})$ reaches $k2^{j+1}$ at or prior to time $\left( k2^j \right)^{\frac{2-\delta}{2(1-\delta)}}$. Based on event $G_{k2^{j+1}}$ the $\lambda_e$ open path reaching $k2^{j+1}$ must intersect a $\lambda_i$-open path that vertically crosses the box $\left[0, k2^{j+1} \right] \times \left[0, \left( k{2^{j+1}} \right)^{\frac{2-\delta}{2(1-\delta)}}\right]$. This then implies that,
    \begin{equation}
        \mathcal{R}(\eta'_t) \geq 0, \text{ for all } t \in \left[  \left( k 2^{j} \right)^{\frac{2-\delta}{2(1-\delta)}},  \left( k 2^{j+1} \right)^{\frac{2-\delta}{2(1-\delta)}}\right].
    \end{equation}
    Using induction we can then deduce that $\limsup_t\mathcal{R}(\eta'_t) = \infty$, and thus $\tau^\emptyset = \infty$.
\begin{figure}[H]
    \centering
    \includegraphics[width=1\linewidth]{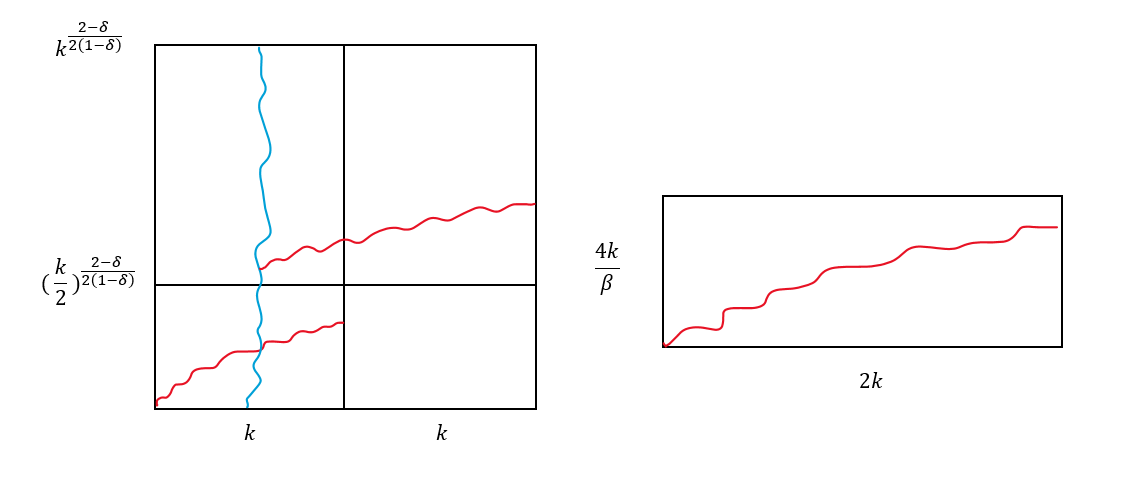}
    \caption{Figures of the event $F_{k}$ and $G_{k}$ occurring on the left, and the right edge having linear speed during an epoch $F_{ik}$ on the right. Paths open under the critical contact process are in blue, while $\lambda_e$-open paths are in red.}
    \label{fig:enter-label}
\end{figure}
Using (\ref{impliessurvival}) we can bound,
\begin{equation}
    \label{bounddksurvival}
    \begin{aligned}
        \mathbb{P}\left( D_k \cap \{ \tau^{\emptyset} < \infty \} \right) & \leq \mathbb{P}\left( D_k \cap \left( \cap_{j=0}^\infty \left( F_{k2^j} \cap G_{k2^j} \right) \right)^c \right) \\
        &\leq \sum_{j=0}^\infty \mathbb{P}\left( F_{k2^j}^c \right) + \mathbb{P}\left( G_{k2^j}^c \right).
    \end{aligned}
\end{equation}
We will now provide an upper bound to $\mathbb{P}(F_k^c)$ for $k \geq 1$. Using the Markov property we note that each $F_{k,n}$ are independent for $n \in \mathbb{Z} \cap  \left[ 0, \dfrac{\beta k^{\frac{\delta}{2(1-\delta)}}}{4} -1 \right] $. Using (\ref{speedpossible}) we have for $n \in \mathbb{Z} \cap \left[ 0, \dfrac{\beta k^{\frac{\delta}{2(1-\delta)}}}{4} -1 \right] $,
\begin{equation}
    \label{fnklikely}
    \mathbb{P}\left( F_{k,n} \right) \geq q' > 0 .
\end{equation}
We note that based on Definition \ref{rightedgeauxillary} whenever $\eta'_{t_{n,k}} = \emptyset$, the process $\eta^{t_{n,k}}_s$ takes instructions from the standard graphical construction starting at time $t_{n,k}$ with no location shift. And thus by independence and (\ref{fnklikely}),
\begin{equation}
    \label{fkthat}
    \mathbb{P}\left(F^c_{k} \right) \leq (1-q')^{\left\lfloor \dfrac{\beta k^{\frac{\delta}{2(1-\delta)}}}{4}  \right\rfloor} .
\end{equation}
And thus combining (\ref{gkthat}), (\ref{bounddksurvival}), and (\ref{fkthat}), there exists universal constants $c_7$ and $c_8 >0$ so that,
\begin{equation}
    \mathbb{P}\left( D_k \cap \{ \tau^\emptyset < \infty \} \right) \leq c_7 \exp\left(-c_8 k^{\frac{\delta}{2}} \right).
\end{equation}
This then completes the lemma.
\end{proof}
\end{lemma}
We will now bound the survival function for $\eta'_t$ when $\eta'_0 = \{0 \}$.
\begin{proof}[Proof of Lemma \ref{survivefunctnbound}]
We can apply Lemma \ref{survivalsomespeed} to obtain universal constants $c_3$, $c_4 >0$ so that,
\begin{equation}
    \label{atandnospeed}
    \mathbb{P}\left( A_t \cap \{ \sup_{0\leq s \leq t} \mathcal{R}(\eta'_s) \geq \frac{t^{1-\frac{\delta}{2}}}{2^{\frac{2-\delta}{2(1-\delta)}}} \} \right) \geq 1 - c_3 \exp\left( - c_4 t^{\frac{\delta}{2}} \right).
\end{equation}
Noting that,
\begin{equation}
    \left( t^{1-\frac{\delta}{2}} \right)^{\frac{2-\delta}{2(1-\delta)}}= t^{1+\frac{\delta^2}{4(1-\delta)}},
\end{equation}
we have for $t \geq 1$,
\begin{equation}
    \label{atdkcontainment}
    A_t \cap \{ \sup_{0\leq s \leq t} \mathcal{R}(\eta'_s) \geq t^{1-\frac{\delta}{2}} \} \subseteq A_t \cap D_{t^{1-\frac{\delta}{2}}}.
\end{equation}
And therefore using Lemma \ref{dkrare}, (\ref{atandnospeed}), (\ref{atdkcontainment}), and that $\delta \in (0,1)$ we have for $t \geq 1$,
\begin{equation}
    \label{survivalfunctionbound}
    \begin{aligned}
    & \mathbb{P}\left( t < \tau^{\emptyset} < \infty \right) \\
    & = \mathbb{P}\left( A_t \cap \{ \tau^\emptyset < \infty \} \right) \\
    &\leq \mathbb{P}\left( A_t \cap \{ \sup_{0 \leq s \leq t} \mathcal{R}(\eta'_s ) < t^{1-\frac{\delta}{2}} \}\right) + \mathbb{P}\left( D_{t^{1-\frac{\delta}{2}}} \cap \{ \tau^\emptyset < \infty \} \right) \\
    &\leq c_3 \exp\left( - c_4 t^{\frac{\delta}{2}} \right) + c_7 \exp\left(-c_8 t^{\frac{\delta(1-\frac{\delta}{2})}{2(1-\delta)}} \right) \\
    & \leq c_{9} \exp\left( - c_{10} t^{\frac{\delta}{4}} \right),
    \end{aligned}
\end{equation}
completing the lemma, with us using on the last inequality that $\delta \in (0,1)$. \\
\end{proof}
\subsection{Proof of Theorems 3 \& 4}
We now generalize the results of Lemma \ref{survivefunctnbound} in order to show Theorems \ref{mainthmsurvivaltime} and \ref{mainthmsurvivalsize}. Theorems \ref{mainthmsurvivaltime} and \ref{mainthmsurvivalsize} are applicable to both $\xi_t$ and $\eta_t$. We will prove each for $\eta_t$ without loss of generality. We start by showing the lemma,
\begin{lemma}
    \label{mixingsetuplemma}
    Let $\{ X_i \}_{i=1}^\infty$ be iid with the law of $\tau^{\emptyset}$, the first extinction time of $\eta'_t$ conditioned on the event $\tau^\emptyset < \infty$. Let $Q$ be an independent geometric random variable with success probability $q''$ independent of $\{ X_i \}_{i=1}^\infty$. Then there exists universal constants $c_{11}$, $c_{12} > 0$ so that for any $n > 0$,
    \begin{equation}
        \mathbb{P}\left( \sum_{i=1}^Q X_i > n \right) \leq c_{11} \exp \left( - c_{12} n^{\frac{\delta}{8}} \right).
    \end{equation}
    \begin{proof}
    Since $q'' = \mathbb{P}\left( \tau^\emptyset = \infty \right) \in (0,1)$ we can apply Lemma \ref{survivefunctnbound} and say there exists a universal constant $c_{13} > 0$ so that for any $n >0$,
    \begin{equation}
    \label{tailobs}
        \mathbb{P}\left( X_i > n \right) \leq c_{13} \exp \left( -c_{10} n^{\frac{\delta}{4}}\right).
    \end{equation}
    Next, we define the events $E_1 = \{ Q \leq n^{\frac{\delta}{4}} \}$ and $E_2 = \{ \max_{1 \leq i \leq n^{\frac{\delta}{4}}} X_i \leq n^{1-\frac{\delta}{4}}$ \}.
    We see by (\ref{tailobs}) that,
    \begin{equation}
    \label{easygeomprobs}
    \begin{aligned}
        \mathbb{P}(E_1) &= 1 -  (1-q'')^{n^{\frac{\delta}{4}}} \\
        \mathbb{P}(E_2) & \geq 1-  n^{\frac{\delta}{4}}c_{13} \exp\left( -c_{10} n^{\frac{\delta}{4} - \frac{\delta^2}{16}}\right).
    \end{aligned}
    \end{equation}
    We see on the event $E_1 \cap E_2$ that,
    \begin{equation}
        \sum_{i=1}^Q X_i \leq n.
    \end{equation}
    Thus since $\delta \in (0,1)$ we complete the lemma by applying a union bound to $\mathbb{P}\left( E_1^c \cup E_2^c \right)$ based on (\ref{easygeomprobs}), and using that $\frac{\delta}{4}- \frac{\delta^2}{16} \geq \frac{\delta}{8}$.
    \end{proof}
\end{lemma}
Let $\eta_t$ be the right edge modified process so that $\mathcal{R}(\eta_0)  \in \mathbb{Z}$. We will now define a recursive process setting initially $T_0 = 0$, and $i = 0$,
        \begin{enumerate}
            \item For $i \geq 0$ let $\tau^{\emptyset,i+1}$ be the first hitting time of $\eta^{T_i}_s$ to the all susceptible state $\emptyset$. 
            \begin{enumerate}
                \item If $\tau^{\emptyset, i+1} = \infty$, set $I = i+1$, $T = T_i$ and terminate the loop
                \item If $\tau^{\emptyset,i+1} < \infty$, increment,
                \begin{equation}
                    T_{i+1} = T_i + \tau^{\emptyset, i+1}.
                \end{equation}
                Additionally set $i = i+1$ and move back to step 1.
            \end{enumerate}
        \end{enumerate}
        We note by the strong Markov property that the set $\{ \tau^{\emptyset, i} \}_{i=1}^{I}$ are each independent. Based on definition \ref{rightedgeauxillary} if $\eta_{T_i} = \emptyset$ for any $i \geq 0$, then $\eta_s^{T_i}$ is the right edge modified contact process with the origin initially infected beginning at time $T_i$. We know that for $i < I$, each $\tau^{\emptyset,i}$ are distributed via the law of $\tau^{\emptyset}$ conditioned on the event $\{ \tau^{\emptyset} < \infty \}$, while via the rejection method for sampling, $\eta^T_s$ has the law of the right edge modified contact process conditioned on the event of survival. We know the event of survival has probability $q'' > 0$. Based on this and Lemma \ref{mixingsetuplemma} we can then conclude,
        \begin{equation}
        \label{tgreaterthann}
            \mathbb{P}\left( T > n \right) \leq c_{11} \exp \left( -c_{12} n^{\frac{\delta}{8}}\right) .
        \end{equation}
We will now prove Theorem \ref{mainthmsurvivaltime}.
\begin{proof}[Proof of Theorem \ref{mainthmsurvivaltime}]
Let $T$ be defined as in section 4.3, and let $\eta_t$ be the right edge modified contact process with $|\eta_0| < \infty$. We note by attractiveness at the right edge, if $\tau^\emptyset > t$, and $T \leq t$, we can conclude that $\tau^\emptyset = \infty$ almost surely. Using this we have,
\begin{equation}
    \label{survivalexpressionthm1}
    \begin{aligned}
        \mathbb{P}\left( t \leq \tau^\emptyset < \infty \right) &\leq \mathbb{P}\left( t < \tau^\emptyset < \infty, \ T \leq t  \right) + \mathbb{P}\left( T > t \right) \\
        & \leq 0 + c_{11} \exp\left(-c_{12} t^{\frac{\delta}{8}} \right).
    \end{aligned}
\end{equation}
\end{proof}
We will leverage a coupling formed by Liggett and later stated in \cite{durrett1983supercritical} to prove Theorem \ref{mainthmsurvivalsize}. Liggett's observation shows if $\{ \zeta_t \}_{t \geq 0}$ is the one dimensional contact process with infection rate $\lambda$ and $|\zeta_0 | = n $, we can form a coupling with $\zeta^{[0,n]}$, a copy of the contact process with $[0,n] \cap \mathbb{Z}$ initially infected so that $| \zeta_t| \geq |\zeta^{[0,n]}_t|$ for all $t \geq 0$. As a consequence we have for $t \geq 0$,
\begin{equation}
\label{liggettcouplingcorrolary}
    \mathbb{P}\left( \zeta_t \neq \emptyset\right)  \leq \mathbb{P}\left( \zeta^{[0,n]} \neq \emptyset\right).
\end{equation}
We will now prove Theorem \ref{mainthmsurvivalsize}.
\begin{proof}[Proof of Theorem \ref{mainthmsurvivalsize}]
We suppose that $\eta_t$ is a copy of the right edge modified contact process with $|\eta_0| = n$. We use again that $\tau^\emptyset= \infty$ almost surely on the event that $\tau^\emptyset > n$ and $T \leq n$. We write,
\begin{equation}
\label{survivalcoupling}
\begin{aligned}
    \mathbb{P}\left( \tau^{\emptyset} = \infty \right) &\geq \mathbb{P}\left( \tau^\emptyset > n, \ T \leq n \right) \\
    & \geq 1 - \mathbb{P}\left( \tau^\emptyset \leq n \right) - \mathbb{P}\left( T > n\right).
\end{aligned}
\end{equation}
By domination $|\eta_t| \geq |\zeta_t|$, where $\zeta_t$ is the critical contact process with $\zeta_0 = \eta_0$. Using $|\zeta_0| =n$, (\ref{liggettcouplingcorrolary}), (\ref{survivalcoupling}), and the same method as Lemma \ref{boxcrossingproportion} we can write,
\begin{equation}
\label{criticalsurvivellinear}
        \mathbb{P}\left( \tau^\emptyset > n \right) \geq \mathbb{P}\left( \zeta_n \neq \emptyset \right) \geq 1 - c_5 \exp\left(-c_6 n^{\delta} \right).
\end{equation}
Hence by (\ref{tgreaterthann}), (\ref{survivalcoupling}), and (\ref{criticalsurvivellinear}) we have,
\begin{equation}
    \mathbb{P}\left( \tau^\emptyset < \infty\right) \leq c_{11} \exp\left( -c_{12} n^{\frac{\delta}{8}}\right) + c_5 \exp\left( - c_6 n^{\delta} \right),
\end{equation}
completing the theorem.
\end{proof}
\section{Law of Large Numbers for $\mathcal{R}(\eta_t)$}
In this section we will form a law of large numbers for the right edge $\mathcal{R}(\eta_t)$ when $\eta_0 \in \Sigma^{\ominus}$. As before we will assume throughout this section that $\lambda_i = \lambda_c$ and $\lambda_e = \lambda_c + \varepsilon$ for some $\varepsilon > 0$. We also note that when $\eta_0 \in \Sigma^{\ominus}$, then $\xi_t$ and $\eta_t$ will have the same dynamics, since $\mathcal{L}(\xi_t) = - \infty$ for all $t$.
\subsection{Bounding Increments of $\mathcal{R}(\eta_t)$ under $\tilde{\mu}$}
We let $\tilde{\eta}_t$ denote the right edge modified contact process with infection rates $\lambda_i = \lambda_c$ and $\lambda_e = \lambda_c + \varepsilon$ for $\varepsilon > 0$, with $\tilde{\eta}_0$ sampled from $\tilde{\mu}$. Note that $\mathcal{R}(\tilde{\eta}_0) = 0$. We will use a similar method of Galves and Presutti in \cite{galves1987edge} to bound the increments of $\mathcal{R}(\tilde{\eta}_t)$. We set $q'' >0$ to be the survival probability of $\eta'_t$, the right edge modified contact process with the origin initially infected.

We will now form a large deviation inequality for $\mathcal{R}(\tilde{\eta}_t)$ for any $t > 0$. 
\begin{lemma}
    \label{largerdeviationincrement}
    Let $\tilde{\eta}_t$ be the right edge modified contact process with infection rates $\lambda_i = \lambda_c$ and $\lambda_e = \lambda_c + \varepsilon$ with $\tilde{\eta}_0$ sampled from $\tilde{\mu}$ the invariant measure from the right edge. Then there exists universal constants $c_{14}>0$ so that for any $n > 0$,
    \begin{equation}
        \mathbb{P}\left( \sup_{0 \leq s \leq t} \left| \mathcal{R}(\tilde{\eta}_s) \right| > 4 \left(\lambda_c + \varepsilon \right)\left( t + n \right)\right) \leq c_{14} \exp\left( - c_{12} n^{\frac{\delta}{8}} \right).
    \end{equation}
    \begin{proof}
        We use that $\mathcal{R}(\tilde{\eta}_t)$ is equal in distribution to $\mathcal{R}(\tilde{\eta}_{t+n}) - \mathcal{R}(\tilde{\eta}_n)$. We will define the random variables $T, I, \{\tau_i\}_{i=1}^I, \{\tau^{\emptyset,i}\}_{i=1}^I, \{ \{\eta^i_s\}_{s \geq 0 } \}_{i=1}^I$ as in section 5.3. Using the attractiveness property at the right edge, for all $s > 0$ we can couple $\mathcal{R}(\tilde{\eta}_{T+s}) - \mathcal{R}(\tilde{\eta}_T)$ with $\mathcal{R}(\eta^T_s)$. We define the event $D = \{ \eta'_t \text{ survives} \}$, and note that $\mathbb{P}(D) = q'' > 0$. We define the event $A$ as the event where Richardson's model (the model with no sites recovering) beginning at the origin and moving in the backwards direction exceeds a distance of $2(\lambda_e + \varepsilon)(n+t)$ at or before time $(n+t)$. We note that,
        \begin{equation}
        \label{richardsonbound1}
             \left\{ \inf_{0 \leq s \leq t + n}  \mathcal{R}(\eta'_s) < -2(\lambda_c + \varepsilon)( t + n) \right\} \cap D \subseteq  A \cap D,
        \end{equation}
        where $\{ \eta'_s\}_{s \geq 0}$ is the right edge modified contact process with a single infection at the origin. We now also use the bound that if $Y \sim \text{Pois}(\lambda)$ then for $a > 0$,
\begin{equation}
    \label{poistailbound}
    \mathbb{P}\left( Y > \lambda + a \right) \leq \exp\left( -\frac{a^2}{2(\lambda + a/3)} \right).
\end{equation}
        Using (\ref{richardsonbound1}), (\ref{poistailbound}), and that $\mathbb{P}(D) = q''$ we have,
        \begin{equation}
        \label{backwardsincrementcareful}
        \begin{aligned}
            &\mathbb{P}\left(\inf_{0 \leq s \leq t+n} \mathcal{R}(\eta^T_s) < -2(\lambda_c + \varepsilon)( t + n)  \right) \\
            &=\mathbb{P}\left( \inf_{0 \leq s < t + n} \mathcal{R}(\eta'_s) <-2(\lambda_c + \varepsilon)( t + n) \ | \ D \right) \\
            &= (q'')^{-1} \mathbb{P}\left( \left\{ \inf_{0 \leq s \leq t + n} \mathcal{R}(\bar{\eta}_s) < -2(\lambda_c + \varepsilon)( t + n) \right\} \cap D \right) \\
            & \leq (q'')^{-1} \mathbb{P}\left( A \cap D \right) \\
            &\leq (q'')^{-1} \exp \left( -\frac{3(t+n)(\lambda_c + \varepsilon)}{8} \right).
        \end{aligned}
        \end{equation}
We now similarly use that $\{\mathcal{R}(\bar{\eta}_t)\}_{t \geq 0}$ is dominated by a rate $\lambda_c + \varepsilon$ Poisson process. Using (\ref{poistailbound}) we obtain,
\begin{equation}
    \label{forwardincrementlabel}
    \begin{aligned}
    &\mathbb{P}\left( \sup_{0 \leq s \leq t + n } \mathcal{R}(\eta^T_s) > 2(\lambda_c + \varepsilon)(t + n)\right) \\
    &= \mathbb{P}\left( \sup_{0 \leq s \leq t + n} \mathcal{R}(\bar{\eta}_s) > 2(\lambda_c + \varepsilon)(t + n) \ | \ D \right) \\
    &\leq (q'')^{-1} \mathbb{P}\left( \sup_{0 \leq s \leq t + n} \mathcal{R}(\bar{\eta}_s) > 2(\lambda_c + \varepsilon)(t + n)  \right) \\
    &\leq (q'')^{-1} \exp \left( -\frac{3(t+n)(\lambda_c + \varepsilon)}{8} \right).
    \end{aligned}
\end{equation}
Note on the event $E$ defined as,
\begin{equation}
    \left\{ T \leq n \right\} \cap \left\{ \inf_{0 \leq s \leq t+n } \mathcal{R}(\eta^T_s) \geq -2 (\lambda_c + \varepsilon)(t+n) \right\} \cap \left\{ \sup_{0 \leq s \leq t + n} \mathcal{R}(\eta^T_s) \leq 2 (\lambda_c + \varepsilon)(t+n)\right\},
\end{equation}
we have by attractiveness from the right edge on $E$,
\begin{equation}
    \begin{aligned}
    &\sup_{n \leq s \leq n+t} \left| \mathcal{R}(\tilde{\eta}_s) - \mathcal{R}(\eta_n) \right|  \\
    & \leq \sup_{T \leq s \leq n + t} \mathcal{R}(\tilde{\eta}_s) - \inf_{T \leq s \leq n + t} \mathcal{R}(\tilde{\eta}_s) \\
    & =  \sup_{T \leq s \leq n + t} \left( \mathcal{R}(\tilde{\eta}_s) - \mathcal{R}(\tilde{\eta}_T) \right) - \inf_{T \leq s \leq n + t} \left(\mathcal{R}(\tilde{\eta}_s) - \mathcal{R}(\tilde{\eta}_T) \right) \\
    & = \sup_{0 \leq s \leq n+t-T} \mathcal{R}(\eta^T_s) - \inf_{0 \leq s \leq t+n-T} \mathcal{R}(\eta^T_s) \\
    & \leq \sup_{0 \leq s \leq t+n} \mathcal{R}(\eta^T_s) - \inf_{0 \leq s \leq t+n} \mathcal{R}(\eta^T_s) \\
    & \leq 4( \lambda_e +  \varepsilon)(t+n).
    \end{aligned}
\end{equation}
Using (\ref{tgreaterthann}), (\ref{backwardsincrementcareful}), and (\ref{forwardincrementlabel}) we have that,
\begin{equation}
    \mathbb{P}\left( E \right) \geq 1 - c_{11}\exp\left(-c_{12}n^{\frac{\delta}{8}} \right) - 2 (q'')^{-1}\exp\left(- \frac{3(t+n)(\lambda_c + \varepsilon)}{8} \right),
\end{equation}
which then completes the lemma.
\end{proof}
\end{lemma}
\subsection{Large Deviations for $\mathcal{R}(\tilde{\eta}_t)$}
We will now prove Theorem \ref{mainthmrttail} which provides stretched exponential bounds for the right edge process at stationary $\mathcal{R}(\tilde{\eta}_t)$ from $\alpha t$. We first show the following lemma,
\begin{lemma}    
    \label{rightmeanlemma}
    Let $\tilde{\eta}_t$ be the right edge modified contact process with $\eta_0$ sampled from the invariant measure from the right edge $\tilde{\mu}$. Then $\mathbb{E}(\mathcal{R}(\eta_1)) = \alpha \in \mathbb{R}$ and for all $t > 0$, 
    \begin{equation}
        \mathbb{E}\left( \mathcal{R}(\tilde{\eta}_t) \right) = \alpha t
    \end{equation}
    \begin{proof}
        Showing $\alpha = \mathbb{E}\left( \mathcal{R}(\tilde{\eta}_1)\right) \in \mathbb{R}$ is a direct consequence of Lemma \ref{largerdeviationincrement}. Fix $t > 0$. Using Lemma \ref{largerdeviationincrement} again we can also conclude that the collection $\{ \dfrac{\mathcal{R}(\tilde{\eta}_s)}{s} \}_{s \geq 1}$ is uniformly integrable. By corollary 1.3 of \cite{terra2024dynamic} there exists a random variable $Y$ so that as $s \rightarrow \infty$,
        \begin{equation}
        \label{almostsureergotic}
            \dfrac{\mathcal{R}(\tilde{\eta}_s)}{s} \overset{a.s.}{\rightarrow} Y,
        \end{equation}
        where $\mathbb{E}(Y) = \alpha$. Using (\ref{almostsureergotic}) and uniform integrability we have,
        \begin{equation}
        \label{alphadef1}
            \lim_{n \rightarrow \infty} \dfrac{\mathbb{E}\left( \mathcal{R}(\tilde{\eta}_{nt}) \right)}{nt} = \alpha .
        \end{equation}
        For any $n \in \mathbb{N}$ using that $\tilde{\mu}$ is an invariant measure we have,
        \begin{equation}
        \label{alphadef2}
        \begin{aligned}
            &\mathbb{E}\left( \mathcal{R}(\tilde{\eta}_{nt}) \right) \\
            &= \sum_{i=1}^n\mathbb{E}\left( \mathcal{R}(\tilde{\eta}_{it}) - \mathcal{R}(\tilde{\eta}_{(i-1)t})   \ \right) \\
            &= n \mathbb{E}\left( \mathcal{R}(\tilde{\eta}_t \right)
        \end{aligned}
        \end{equation}
        And thus combining (\ref{alphadef1}) and (\ref{alphadef2}) we have that $\mathbb{E}(\mathcal{R}(\tilde{\eta}_t)) = \alpha t$.
    \end{proof}
\end{lemma}
We now prove Theorem \ref{mainthmrttail}.
\begin{proof}[Proof of Theorem \ref{mainthmrttail}]
Fix $\gamma \in [0,\frac{1}{3})$ and define the constants,
\begin{equation}
    \begin{aligned}
        \gamma_1 &= \frac{\gamma + \frac{1}{3}}{2} \\
        \gamma_2 &= \frac{1}{12} - \frac{\gamma}{4}.
    \end{aligned}
\end{equation}
We note that $\gamma_1 - 2\gamma_2  = \gamma$, and $\gamma_1 < \frac{1}{3}$, and $\gamma_2 > 0$. Let $M_0 = -t^{\gamma_1}$, and set $i = 1$. We will define epochs in time iteratively by the procedure below. 
\begin{enumerate}
    \item When $i \geq 1$ is fixed, set $j = 1$, and $T_{i,0} =  t^{\gamma_1}$. We will now run the procedure iteratively for $j$.
        \item Let $\tau^{i,j}$ be the first hitting time of $\eta^{M_{i-1} +T_{i,j-1}}_s$ to the all susceptible state $\emptyset$. Define the event, $D_{ij} = \{ \tau^{i,j } > t^{\gamma_1} \}$. 
        \begin{enumerate} 
        \item On the event $D_{ij}$ set $M_i = M_{i-1} + T_{i,j-1}$ and $I_{i} = j$. Set $T_{i,I_i} = t^{\gamma_1}$ and $j = 1$, $i = i+1$. Return to step $1$. 
        \item On the event $D_{ij}^c$ set $T_{i,j} = T_{i,j-1} + \tau^{i,j}$, increment $j = j + 1$ and return to step 2.
        \end{enumerate}
    \end{enumerate}
    The random times $ \{ M_i + t^{\gamma_1} \}_{i=1}^\infty$ are each stopping times. This can be seen since $M_1 + t^{\gamma_1}$ is the first time we observe a $\lambda_e$-open path starting from the right edge that survives for time $t^{\gamma_1}$. Similarly, $M_2 + t^{\gamma_1}$ is the first time following $M_1 + t^{\gamma_1}$ where we see a $\lambda_e$-open path starting from the right edge that survives for time $t^{\gamma_1}$. We can apply this reasoning inductively to confirm that each $M_i + t^{\gamma_1}$ is a stopping time for $i \geq 1$.
    
    It also follows that for $i \geq 1$, the random variable $M_i - M_{i-1}$ is a measurable function of the auxiliary processes, 
    \begin{equation}
    \left\{ \left\{ \eta^{M_{i-1}+T_{i,j-1}}_s \right\}_{s \in [0, T_{i,j}]} \right\}_{j=1}^{I_i},     
    \end{equation}
    which are independent of $\tilde{\eta}_{M_{i-1} + t^{\gamma_1}}$, and depends on instructions from the sitewise construction at disjoint times for each $i \in \mathbb{N}$. Hence by the strong Markov property we can conclude that the random variables $\{M_i - M_{i-1} \}_{i=2}^\infty$ are independent and identically distributed, and $M_1 - M_0$ has the law of $M_2 - M_1 - t^{\gamma_1}$ and is independent of $M_i - M_{i-1}$ for all $i \geq 2$.
    
    For $i \in \mathbb{Z}_{0+}$ let $\tau^{i,\emptyset}$ be the first hitting time of $\eta^{M_{i-1}}_s$ to the all susceptible configuration $\emptyset$. We define the events,
    \begin{equation}
    \label{majorevent}
    \begin{aligned}
    A_t &=  \{ \tau^{i,\emptyset} > M_{i} - M_{i-1} \text{ for } 2 \leq i \leq t^{1-\gamma_1} \} \\
    B_t &= \{ \max_{1 \leq i \leq t^{1-\gamma_1}} M_{i} - M_{i-1} \leq t^{\gamma_1} + t^{\gamma_2} \} \\
    C_t &= \{ \max_{2 \leq i \leq t^{1-\gamma_1}} \sup_{0 \leq s \leq M_{i} - M_{i-1}} | \mathcal{R}(\eta^{M_i}_{s }) |\leq 12(\lambda_c + \varepsilon)t^{\gamma_1} \text{, } |\mathcal{R}(\tilde{\eta}_{M_1})| \leq  12( \lambda_c +  \varepsilon)t^{\gamma_1} \}  \\
    E_t & = A_t \cap B_t \cap C_t.
    \end{aligned}
    \end{equation}
    We also define,
    \begin{equation}
        H(t) = \sup \{ i \in \mathbb{N} \ | \ M_{i} \leq t \}.
    \end{equation}
    Using attractiveness from the right edge, we also have that on the event $E_t$,
    \begin{equation}
        \mathcal{R}(\tilde{\eta}_t) = \mathcal{R}(\tilde{\eta}_{M_1}) + \sum_{i=1}^{t^{1-\gamma}} \mathcal{R}(\eta^{M_i}_{M_{i+1} - M_i}) + \mathcal{R}(\eta^{M_{H(t)}}_{t-M_{H(t)}}) - \sum_{i=H(t) + 1}^{t^{1-\gamma}} \mathcal{R}(\eta^{M_i}_{M_{i+1}- M_i}).
    \end{equation}
    Note that $H(t) \geq 1$ on the event $E_t$. We will now bound the probability $\mathbb{P}(E_t)$. By our construction it follows for $i \geq 1$ that $\eta^{M_i}_s$ is a copy of the right edge modified contact process with the origin initially infected with a law conditioned on the event $\{ \tau^{i,\emptyset} > t^{\gamma_1} \}$. Recall for the right edge modified contact process $\eta'_t$ with the origin initially infected,
    \begin{equation}
    \label{increasetosurvival}
        \mathbb{P}(\tau^\emptyset > t^{\gamma_1}) \geq q''.
    \end{equation}
    We can then apply (\ref{increasetosurvival}) and Lemma \ref{survivefunctnbound} to obtain for $i \geq 2$,
    \begin{equation}
    \label{atsetup1}
    \begin{aligned}
    &\mathbb{P}\left( \tau^{i,\emptyset} \leq M_{i} - M_{i-1}  \right) \\
    &\leq \mathbb{P}\left( \tau^{i,\emptyset} < \infty \right) \\
    &\leq (q'')^{-1} c_9 \exp\left( - c_{10} t^\frac{\gamma_1 \delta }{8}\right).
    \end{aligned}
    \end{equation}
    And hence applying a union bound to the bound in (\ref{atsetup1}) gives,
    \begin{equation}
    \label{atlikely}
        \mathbb{P}(A_t^c) \leq t^{1-\gamma_1} (q'')^{-1}  c_9 \exp\left( - c_{10} t^\frac{\gamma_1 \delta }{8}\right).
    \end{equation}
    We will now provide a bound to $\mathbb{P}(B_t)$. We can write for $i \geq 2$
    \begin{equation}
    \label{migeom}
        M_{i} - M_{i-1} - t^{\gamma_1} = \sum_{j=1}^{I_i} \tau^{i,j},
    \end{equation}
    We can apply the same decomposition as (\ref{migeom}) to the quantity $M_1 - M_0$. Using that for $i \geq 1$, $M_i + t^{\gamma_1}$ is a stopping time, by the strong Markov property both $I_i$ and $\sum_{j=1}^{I_i} \tau^{i,j}$ are independent of $\mathcal{F}_{M_{i-1} + t^{\gamma_1}}$. $I_i$ is dominated by a geometric$(q'')$ random variable. By the strong Markov property the collection of variables $\left\{ \tau^{i,j} \right\}_{j=1}^{I_i-1}$ are iid with the law of $\tau^\emptyset$ conditioned on the event $\tau^\emptyset < t^{\gamma_1}$. Using this and the same method as in the proof of Lemma \ref{mixingsetuplemma} we have that for all $i \in \mathbb{N}$ there exists a universal constant $c_{13} > 0$ so that for any $n > 0$,
    \begin{equation}
        \label{boundtauidif}
        \mathbb{P}\left( M_{i} - M_{i-1} > t^{\gamma_1} + n \right) \leq c_{13} \exp\left(- c_{12} n^{\frac{\delta}{8}} \right).
    \end{equation}
    Hence plugging in $n = t^{\gamma_2}$ to (\ref{boundtauidif}) and applying a union bound then gives,
    \begin{equation}
        \label{btlikely}
        \mathbb{P}\left( B_t^c \right) \leq t^{1-\gamma_1} c_{13} \exp \left(-c_{12} t^{\frac{\gamma_2 \delta}{8}} \right).
    \end{equation}
    We will now bound the probability of $C_t^c \cap A_t \cap B_t$. Using attractiveness from the right edge, we can couple $\mathcal{R}(\eta^{M_{i-1}}_s)$ with the right edge of an auxiliary process $\tilde{\eta}^*_s$ for all $s \in [0, M_{i} - M_{i-1} ]$ on the event $A_t \cap B_t$. We define $\tilde{\eta}^*_s$ as a copy of the right edge modified contact process with initial configuration sampled from $\tilde{\mu}$, and a law conditioned on the event that there exists a $\lambda_e$ open path beginning at $0$ until time at least $t^{\gamma_1}$. We note this event has a probability of at least $q''$. On the event $B_t$ we note that $M^{i} - M_{i-1} \leq t^{\gamma_1} + t^{\gamma_2}$. Let $\tilde{\eta}_s$ be the right edge modified contact process sampled via $\tilde{\mu}$ (not conditioned on any event), and $\tilde{\tau}^\emptyset$ be the first time that all $\lambda_e$-open paths of $\tilde{\eta}_s$ die out when beginning at time $0$ and site $0$. By Lemma \ref{largerdeviationincrement} we have for $2 \leq i \leq t^{1-\gamma_1}$ and $t \geq 1$, using that $\gamma_1 > \gamma_2$,
    \begin{equation}
    \label{ctlikelysetup}
    \begin{aligned}
        & \mathbb{P}\left( A_t \cap B_t \cap \left\{  \sup_{0 \leq s \leq M_{i} - M_{i-1} }\left| \mathcal{R}(\eta^{M_{i-1}}_{s })\right| > 12(\lambda_c + \varepsilon)t^{\gamma_1} \right\} \right) \\
        &\leq \mathbb{P}\left( \sup_{0 \leq s \leq t^{\gamma_1} + t^{\gamma_2}} \left| \mathcal{R}(\tilde{\eta}^*_s ) \right| > 12(\lambda_c + \varepsilon)t^{\gamma_1} \right) \\
        &= \mathbb{P}\left(  \sup_{0 \leq s \leq t^{\gamma_1} + t^{\gamma_2}} \left| \mathcal{R}(\tilde{\eta}_s ) \right| > 12( \lambda_c + \varepsilon)t^{\gamma_1} \ | \ \tilde{\tau}^\emptyset > t^{\gamma_1}\right) \\
        & \leq \mathbb{P}(\tilde{\tau}^\emptyset > t^{\gamma_1} )^{-1} \mathbb{P}\left( \sup_{0 \leq s \leq t^{\gamma_1} + t^{\gamma_2}}  \left| \mathcal{R}(\tilde{\eta}_s) \right| > 12(\lambda_c + \varepsilon)t^{\gamma_1} \right)  \\
        & \leq \mathbb{P}(\tilde{\tau}^\emptyset > t^{\gamma_1} )^{-1} \mathbb{P}\left( \sup_{0 \leq s \leq 2t^{\gamma_1} }  \left| \mathcal{R}(\tilde{\eta}_s) \right| > 12(\lambda_c + \varepsilon)t^{\gamma_1} \right)  \\
        &\leq (q'')^{-1} c_{14} \exp \left( -c_{12} t^{\frac{\gamma_1 \delta}{8}} \right).
    \end{aligned}
    \end{equation}
    Using same approach as in (\ref{ctlikelysetup}) and applying Lemma \ref{largerdeviationincrement} can be used to show when $t \geq 1$,
    \begin{equation}
        \label{ctlikelysetup2}
        \mathbb{P}\left( A_t \cap B_t \cap \left\{  \sup_{0 \leq s \leq M_{1} - M_{0} }\left| \mathcal{R}(\tilde{\eta}_{s })\right| > 12( \lambda_c + \varepsilon)t^{\gamma_1} \right\} \right) \leq c_{14} \exp \left( -c_{12}  t^{\frac{\gamma_1 \delta}{8}} \right).
    \end{equation}
    Applying a union bound with the results from (\ref{ctlikelysetup}) and (\ref{ctlikelysetup2}) yields,
    \begin{equation}
    \label{ctlikely}
        \mathbb{P}(C_t^c \cap A_t \cap B_t ) \leq t^{1-\gamma_1} (q'')^{-1} c_{14} \exp \left( -c_{12} t^{\frac{\gamma_1 \delta}{8}} \right).
    \end{equation}
    And thus applying (\ref{atlikely}), (\ref{btlikely}), and (\ref{ctlikely}), there exists universal constants $c_{15}$ and $c_{16} > 0$ so that for any $t \geq 1$,
    \begin{equation}
    \begin{aligned}
        \label{petbound}
        \mathbb{P}\left( E_t \right) & \geq 1 - \mathbb{P}(A_t^c) - \mathbb{P}(B_t^c) - \mathbb{P}(C_t^c \cap A_t \cap B_t ) \\
        & \geq 1- c_{15} \exp\left( -c_{16} t^{\frac{\gamma_2 \delta}{8}} \right).
    \end{aligned}
    \end{equation}
    We now write,
    \begin{equation}
    \label{rightmean1}
    \begin{aligned}
        \mathcal{R}(\tilde{\eta}_t) = &\mathbbm{1}\{E_t\} \left(\mathcal{R}(\tilde{\eta}_{M_1}) + \sum_{i=2}^{t^{1-\gamma_1}} \mathcal{R}(\eta^{M_{i-1}}_{M_i 
 - M_{i-1}}) + \mathcal{R}(\eta^{M_{N(t)+1}}_{t - \tau_{N(t)}}) - \sum_{i=N(t)+1}^{t^{1-\gamma_1}}  \mathcal{R}(\eta^{M_{i-1}}_{M_i 
 - M_{i-1}}) \right) + \\
 &\mathbbm{1}\{ E_t^c \} \mathcal{R}(\tilde{\eta}_t).
    \end{aligned}
    \end{equation}
    \label{rightmean2}
    By Lemma \ref{rightmeanlemma} we have that $\mathbb{E}\left( \mathcal{R}(\tilde{\eta}_t) \right) = \alpha t$ for $\alpha \in \mathbb{R} $. We also note on the event $E_t$ that pointwise,
    \begin{equation}
        \label{htbound}
        H(t) \geq \left\lfloor t^{1-\gamma_1}(1+t^{\gamma_2 - \gamma_1})^{-1} \right\rfloor \geq t^{1-\gamma_1}(1 -  t^{\gamma_2 - \gamma_1}) -1.
    \end{equation}
    Using (\ref{majorevent}), the following three inequalities hold on the event $E_t$,
    \begin{equation}
        \label{consequenceet}
        \begin{aligned}
             \left| \mathcal{R}(\eta^{M_{i-1}}_{M_i - M_{i-1}}) \right| &\leq 12(\lambda_c +  \varepsilon)t^{\gamma_1} \text{ for } 2 \leq i \leq t^{1-\gamma_1} \\
             \left| \mathcal{R}(\eta^{M_{H(t)}}_{t-M_{H(t)}}) \right| &\leq 12( \lambda_c +  \varepsilon)t^{\gamma_1} \\
             \left| \mathcal{R}(\tilde{\eta}_{M_1}) \right| &  \leq 12( \lambda_c +  \varepsilon)t^{\gamma_1}
        \end{aligned}
    \end{equation}
    Applying (\ref{htbound}) and (\ref{consequenceet}) along with the triangle inequality then gives when $t \geq 1$ on $E_t$,
    \begin{equation}
        \label{rightmean3}
        \begin{aligned}
        \left| \mathcal{R}(\tilde{\eta}_{M_1})+ \mathcal{R}(\eta^{M_{H(t)}}_{t - M_{H(t)}}) - \sum_{i=H(t)+1}^{t^{1-\gamma_1}}  \mathcal{R}(\eta^{M_{i-1}}_{M_i 
 - M_{i-1}} )\right| &\leq 12( \lambda_c +  \varepsilon)t^{\gamma_1}( t^{1+\gamma_2 - 2 \gamma_1} + 3) \\
        & \leq 48(1 + 2 \lambda_c + 2 \varepsilon)t^{1 + \gamma_2 -  \gamma_1}
        \end{aligned}
    \end{equation}
    Using Lemma \ref{largerdeviationincrement} and (\ref{petbound}) we have for $t$ large,
    \begin{equation}
    \label{rightmean6}
        \begin{aligned}
        &\left | \mathbb{E}\left( \mathbbm{1} \{ E_t^c\} \mathcal{R}(\tilde{\eta}_t)\right) \right| \\
        &\leq \int_0^{t^2} \mathbb{P}(E_t^c) dn + \int_{t^2}^\infty \mathbb{P}\left( \left| \mathcal{R}(\tilde{\eta}_t) \right| > n \right) dn \\
        & \leq t^2 c_{15} \exp\left(c_{16} t^{\frac{\gamma_2 \delta}{4}} \right) +  \int_{t^2}^\infty c_{14} \exp\left( -c_{12} \left( \frac{ n - 4(\lambda_c +  \varepsilon)(t^{\gamma_1} + t^{\gamma_2})  }{ 4( \lambda_c +  \varepsilon)} \right)^{\frac{\delta}{8}}\right) dn \\
        &= o(1).
        \end{aligned}
    \end{equation}
    And therefore, by (\ref{rightmean1}), (\ref{rightmean3}), (\ref{rightmean6}), and Lemma \ref{rightmeanlemma} we have for $t$ large,
    \begin{equation}
    \label{meanoff}
    \mathbb{E}\left( \mathbbm{1} \{ E_t \}\sum_{i=2}^{t^{1-\gamma_1}} \mathcal{R}(\eta^{M_{i-1}}_{M_i - M_{i-1}}) \right) = \alpha t + \mathcal{O}(t^{1-\gamma_1 + \gamma_2}).
    \end{equation}
    For $2 \leq i \leq t^{1-\gamma_1}$ define the event $E_{it}$ as the event,
    \begin{equation}
    \label{wehavetherightmean}
        \{ \tau^{i, \emptyset} > M_{i} - M_{i-1}  \} \cap \{ M_i - M_{i-1} \leq t^{\gamma_1} + t^{\gamma_2} \} \cap \{ \sup_{0 \leq s \leq M_i - M_{i-1}} \left| \mathcal{R}(\eta^{M_{i-1}}_s) \right| \leq 12(\lambda_c + \varepsilon)t^{\gamma_1}\}.
    \end{equation}
    
    Using that the variables $\{ M_i + t^{\gamma_1} \}_{i=1}^\infty$ are stopping times and the strong Markov property, one can show the collection of variables $\left\{  \mathcal{R}(\eta^{M_{i-1}}_{M_i - M_{i-1}}) \mathbbm{1}\{E_{it} \} \right\}_{i=2}^{t^{1-\gamma_1}} $ are independent and identically distributed. We also note that almost surely,
    \begin{equation}
    \label{hoef1}
        \left| \mathcal{R}(\eta^{M_{i-1}}_{M_i - M_{i-1}}) \mathbbm{1}\{E_{it} \} \right| \leq 12(\lambda_c + \varepsilon)t^{\gamma_1}.
    \end{equation}
    Using (\ref{petbound}) and (\ref{hoef1}), we can bound,
    \begin{equation}
    \label{hoef2}
        \mathbb{E}\left( \mathcal{R}(\eta^{M_{i-1}}_{M_i - M_{i-1}}) \mathbbm{1}\{E_{it} \} \mathbbm{1}\{E_t^c \} \right) \leq  \mathcal{O}\left( t^{\gamma_1} \exp\left( -c_{16} t^{\frac{\gamma_2 \delta}{8}} \right) \right).
    \end{equation}
    And thus by (\ref{meanoff}) and (\ref{hoef2}) we have for $2 \leq i \leq t^{1-\gamma_1}$,
    \begin{equation}
    \label{hoef3}
    \mathbb{E}\left( \mathcal{R}(\eta^{M_{i-1}}_{M_i - M_{i-1}}) \mathbbm{1}\{E_{it} \} \right) = \alpha t^{\gamma_1}\left( 1 + \mathcal{O}(t^{\gamma_2 - \gamma_1})\right).
    \end{equation}
    Let $h(t) = \mathbb{E}\left( \sum_{i=2}^{t^{1-\gamma_1}} \mathcal{R}(\eta^{M_{i-1}}_{M_i - M_{i-1}}) \mathbbm{1}\{E_{it} \} \right) $. Using (\ref{hoef1}) and (\ref{hoef3}) we have by Hoeffding's inequality for $b > 0$,
    \begin{equation}
    \label{actuallyhoef}
    \begin{aligned}
        \mathbb{P}\left( \left| \sum_{i=2}^{t^{1-\gamma_1}} \mathcal{R}(\eta^{M_{i-1}}_{M_i - M_{i-1}}) \mathbbm{1}\{E_{it} \} - h(t) \right| > \frac{bt^{1-\gamma_1 + 2\gamma_2 }}{2} \right) &\leq \exp \left( -2 \frac{\frac{1}{4}b^2 t^{2-2\gamma_1 + 4 \gamma_2}}{t^{1-\gamma_1}\left(12(\lambda_c + \varepsilon)t^{\gamma_1} \right)^2} \right) \\
        &= \exp\left(-\frac{b^2 t^{1-3\gamma_1 + 4 \gamma_2}}{288 (\lambda_c + \varepsilon)^2} \right)
    \end{aligned}
    \end{equation}
    Define the event,
        \begin{equation}
            \begin{aligned}
                F &= \left\{  \left| \sum_{i=2}^{t^{1-\gamma_1}} \mathcal{R}(\eta^{M_{i-1}}_{M_i - M_{i-1}}) \mathbbm{1}\{E_{it} \} - h(t) \right| \leq \frac{ b t^{1-\gamma_1 + 2\gamma_2 }}{2}   \right\}  .
            \end{aligned}
        \end{equation}
        We have by (\ref{actuallyhoef}) that,
        \begin{equation}
        \label{eventshighenough}
        \begin{aligned}
            \mathbb{P}(F) &\geq 1 -\exp\left(-\frac{b^2 t^{1-3\gamma_1 + 4 \gamma_2}}{288 (\lambda_c + \varepsilon)^2} \right).
        \end{aligned}
        \end{equation}
        Note that $1- 3 \gamma_1 + 4 \gamma_2 > 0$. Thus we have by (\ref{eventshighenough}) and (\ref{petbound}) that,
        \begin{equation}
        \label{probabilityforlargedeviation}
            \mathbb{P}(E_t \cap F) \geq 1 -\exp\left(-\frac{b^2 t^{1-3\gamma_1 + 4 \gamma_2}}{288 (\lambda_c + \varepsilon)^2} \right) - c_{15} \exp\left(-c_{16}t^{ \frac{\gamma_2 \delta }{8}} \right).
        \end{equation}
        We also have by (\ref{rightmean3}) for $t > 0$ sufficiently large on the event $E_t \cap F$,
        \begin{equation}
        \label{boundthejunk}
             \left| \mathcal{R}(\tilde{\eta}_{M_1}) +  \mathcal{R}(\eta^{M_{H(t)}}_{t - M_{N(t)}}) - \sum_{i=H(t)+1}^{t^{1-\gamma_1}}  \mathcal{R}(\eta^{M_{i-1}}_{M_i 
 - M_{i-1}} )\right| \leq 48( \lambda_c + \varepsilon) t^{1-\gamma_1 + \gamma_2}  \leq \frac{b t^{1-\gamma_1 + 2\gamma_2 }}{4}.
        \end{equation}
        We know by (\ref{meanoff}) that $h(t) = \alpha t + \mathcal{O}(t^{1-\gamma_1 + \gamma_2})$. Thus for $t$ large enough we can conclude,
        \begin{equation}
            \label{meanclosetob}
            \left|h(t) - \alpha t \right| \leq \frac{bt^{1-\gamma_1 + 2 \gamma_2}}{4}.
        \end{equation}
       Using that $E_t \subseteq E_{it}$ for $2 \leq i \leq t^{1-\gamma_1}$, the decomposition of $\mathcal{R}(\tilde{\eta}_t)$ in (\ref{rightmean1}), (\ref{boundthejunk}), (\ref{meanclosetob}), and the triangle inequality we have shown that on the event $E_t \cap F $ for $t$ sufficiently large,
       \begin{equation}
           \left| \mathcal{R}(\tilde{\eta}_t) - \alpha t \right| \leq b t^{1 - \gamma_1 + 2\gamma_2 } = bt^{1-\gamma}.
       \end{equation}
       Using this and (\ref{probabilityforlargedeviation}) we have completed the theorem.
       \end{proof}
    \subsection{Strong Law of Large Numbers}
    We will now form the strong law of large numbers for $\dfrac{\mathcal{R}(\eta_t)}{t}$ when $\eta_0$ is set to be in the half-space $\Sigma^{\ominus}$. \\
    \begin{proof}[Proof of Theorem 4]
            We will first show the strong law of large numbers for $\tilde{\eta}_t$, the right edge modified process sampled from the invariant measure from the right edge $\tilde{\mu}$, and then generalize the results to any $\eta_0 \in \Sigma^{\ominus}$. We also will show that $\alpha \geq \varepsilon$ in section 6.4.   
            
            By Theorem \ref{mainthmrttail} there exists constants $c, \ c', \  a>0$ so that,
            \begin{equation}
                \mathbb{P}\left( | \mathcal{R}(\tilde{\eta}_t) - \alpha t | > t^{\frac{3}{4}}\right) \leq c \exp \left( -c' t^a \right).
            \end{equation}
            Thus by the Borel Cantelli lemma we know,
            \begin{equation}
            \label{lln1}
                \mathbb{P}\left( |\mathcal{R}(\tilde{\eta}_n) - \alpha n| >  t^{\frac{3}{4}} \text{ for } n \in \mathbb{N}\ \text{i.o.}\right) = 0. 
            \end{equation}
        Using that $\mathcal{R}(\tilde{\eta}_t)$ has stationary increments, we can apply Lemma \ref{largerdeviationincrement} and the Borel Canteilli lemma to show,
        \begin{equation}
        \label{lln2}
            \mathbb{P}\left( \sup_{0 \leq s \leq 1} \left|\mathcal{R}(\tilde{\eta}_{n+s}) - \mathcal{R}(\tilde{\eta}_n) \right| > n^{\frac{3}{4}} \text{ for } n \in \mathbb{N} \text{ i.o.}\right) = 0.
        \end{equation}
        And therefore combining (\ref{lln1}) and (\ref{lln2}) we have,
        \begin{equation}
            \mathbb{P}\left(  \left| \mathcal{R}(\tilde{\eta}_t) - \alpha t \right| < 2t^{\frac{3}{4}} \text{ eventually}\right) =1,
        \end{equation}
        which completes the law of large numbers when $\eta_0 \sim \tilde{\mu}$.

        Suppose now that $\{ \eta_t \}_{t \geq 0}$ is the right edge modified contact process so that $\eta_0 \in \Sigma^{\ominus}$.Let the random variable $T$ be defined as in section 4.3. When $T \leq t$ we can write,
        \begin{equation}
        \mathcal{R}(\eta_t) = \mathcal{R}(\eta_T) + \mathcal{R}(\eta^T_{t-T}).
        \end{equation}
        Using attractiveness from the right edge we know that $ \{ \mathcal{R}(\eta^T_{s-T}) \}_{s \geq T}$ has the law of $ \{ \mathcal{R}(\tilde{\eta}_s) \}_{t \geq 0}$ conditioned on the event that the origin yields an infinite $\lambda_e$ open path beginning at time $0$. Since this event has probability $q'' > 0$, and both $T< \infty$ and $| \mathcal{R}(\eta_T)|< \infty$ almost surely, we have completed the theorem.
    \end{proof}
    \subsection{Lower Bound to Edge Speeds}
    We will now show that when $\lambda_i = \lambda_c$, and $\lambda_e = \lambda_c + \varepsilon$, then the edge speed $\alpha \geq \varepsilon$. Let $\bar{\eta}_t$ be a copy of the right edge modified contact process with $(-\infty, 0]$ initially infected, and $\bar{\zeta}_t$ be the critical contact process with $(-\infty, 0]$ with $\bar{\eta}_t$ and $\bar{\zeta}_t$ defined under the same graphical representation. By Theorem \ref{mainthmlln} we have as $t \rightarrow \infty$,
    \begin{equation}
    \label{etabarconvergeas}
        \frac{\mathcal{R}(\bar{\eta}_t)}{t} \overset{a.s}{\rightarrow} \alpha.
    \end{equation}
    We also note that the collection, $\left\{ \dfrac{\mathcal{R}(\bar{\eta}_t)}{t} \right\}_{t \geq 1}$ is uniformly integrable. To see this we note that $\mathcal{R}(\bar{\eta}_t)^+$ is dominated by a rate $\lambda_c + \varepsilon$ Poisson process. We also see that $\dfrac{\mathcal{R}(\bar{\eta}_t)^-}{t}$ is dominated by $\dfrac{\mathcal{R}(\bar{\zeta}_t)^-}{t}$, which we can show is uniformly integrable for $t \geq 1$ by Lemma \ref{boxcrossrt}. And hence we can conclude that as $t \rightarrow \infty$,
    \begin{equation}
        \label{speedl1convergence}
        \mathbb{E}\left( \dfrac{\mathcal{R}(\bar{\eta}_t)}{t} \right) \rightarrow \alpha .
    \end{equation} 
    Using Lemma \ref{boxcrossrt} again we can conclude that for $t$ large,
    \begin{equation}
    \label{criticalcontactmean}
        \mathbb{E}\left( \mathcal{R}(\bar{\zeta}_t) \right) = \mathcal{O}(t^{1-\delta}).
    \end{equation}
    We will now give a proposition from \cite{liggett1985interacting},
    \begin{lemma}[Liggett]
    \label{liggettedgespeedsetup}
    For any set $A \subseteq \mathbb{Z}$, let $\zeta^A_t$ be a copy of the contact process with infection rate $\lambda$ and initial infected set $A$. Suppose $B \subset A$, where $A$ and $B$ are both in $ \Sigma^{\ominus}$, and let $C$ be any finite set. Then under the standard graphical representation of the contact process,
    \begin{equation}
        0 \leq \mathcal{R}(\zeta^{A \cup C}_t) - \mathcal{R}(\zeta^A_t) \leq \mathcal{R}(\zeta^{B \cup C}_t) -  \mathcal{R}(\zeta^B_t).
    \end{equation}
    In particular, for $B \subset (-\infty, -1]$,
    \begin{equation}
        \mathbb{E}\left( \mathcal{R}(\zeta^{B \cup \{0\}}_t) - \mathcal{R}(\zeta^{B}_t) \right) \geq 1.
    \end{equation}
    \end{lemma}
     For $i \in \mathbb{N}$ we let $\bar{\eta}_{i,t}$ be a copy of the right edge modified contact process that takes edge boosts from the clock $N_{e2}(t)$ up until time $X_i = \inf \{ t \geq 0 \ | \ N(t) \geq i \}$, and after runs via the instructions of the critical contact process. Recall $\mathcal{F}_t$ corresponds to the filtration of the sitwise construction of the border modified contact process up until time $t \geq 0$. For notation let $\zeta^C_t$ correspond to the critical contact process with initial infected region $C$, and let $C_i$ be the infected regin of $\bar{\eta}_{i,X_i}$, and $C_i'$ be the infected region of $\bar{\eta}_{i-1,t}$. Using the strong Markov property we can write,
     \begin{equation}
     \label{edgespeeedsetup1}
     \begin{aligned}
         & \mathbb{E}\left( \mathcal{R}(\bar{\eta}_{i,t}) - \mathcal{R}(\bar{\eta}_{i-1,t}) \right) \\
        &= \mathbb{E}\left( \mathbbm{1} \{ X_i \leq t \} \mathbb{E} \left( \mathcal{R}(\bar{\eta}_{i,t})  - \mathcal{R}(\bar{\eta}_{i-1,t}) \mid \mathcal{F}_{X_i} \right) \right) \\
        &= \mathbb{E}\left( \mathbbm{1} \{ X_i \leq t \} \mathbb{E} \left( \mathcal{R}(\zeta^{C_i}_{t-X_i}) - \mathcal{R}(\zeta^{C_i'}_{t-X_i}) \right) \right) 
        \end{aligned}
        \end{equation}
        Note that $C_i' \subset C_i$ and $C_i \backslash C_i' = \{ \mathcal{R}(\bar{\eta}_{X_i} \}$, and for all $s \geq X_i$, both $\bar{\eta}_{i,s}$ and $\bar{\eta}_{i-1,s}$ will evolve via the critical contact process. Using this, Lemma \ref{liggettedgespeedsetup} and integrating over the random set $C_i$ we have,
        \begin{equation}
        \label{edgespeedsetup1point5}
        \begin{aligned}
        \mathbb{E}\left( \mathbbm{1} \{ X_i \leq t \} \mathbb{E} \left( \mathcal{R}(\zeta^{C_i}_{t-X_i}) - \mathcal{R}(\zeta^{C_i'}_{t-X_i}) \right) \right) 
        &\geq \mathbb{E}\left( \mathbbm{1}\{ X_i \leq t \} (1) \right) \\
        &= \mathbb{P}\left(N_{e2}(t) \geq i \right).
     \end{aligned}    
     \end{equation}
    Using (\ref{poistailbound}), (\ref{criticalcontactmean}), (\ref{edgespeeedsetup1}), (\ref{edgespeedsetup1point5}), and that $N_{e2}(t) \sim \text{Pois}\left( (\lambda_c + \varepsilon )t \right)$, for $t$ large we have,
    \begin{equation}
    \label{edgespeeedsetup2}
    \begin{aligned}
        & \mathbb{E}\left( \mathcal{R}(\bar{\eta}_{t}) \right) \\
        &= \mathbb{E}\left( \mathcal{R}(\bar{\zeta}_t) + \left(\mathcal{R}(\bar{\eta}_t) - \mathcal{R}(\bar{\eta}_{2 \varepsilon t, t}) \right) +   \sum_{i=1}^{2 \varepsilon t} \mathcal{R}(\bar{\eta}_{i,t}) - \mathcal{R}(\bar{\eta}_{i-1,t})\right) \\
        &\geq \mathcal{O}(t^{1-\delta}) + 0 + \sum_{i=1}^{2 \varepsilon t} \mathbb{P}\left( N_{e2}(t) \geq i \right) \\
        &\geq \varepsilon t + \mathcal{O}(t^{1-\delta}).
    \end{aligned}
    \end{equation}
    And hence by (\ref{speedl1convergence}), (\ref{edgespeeedsetup2}) and limiting $t \rightarrow \infty$, we can conclude $\alpha \geq \varepsilon$.

\section{Brownian Motion Approximation}
\subsection{Setup}
We will now show a functional central limit theorem for the right edge process $\mathcal{R}(\eta_t)$ whenever $\eta_0$ is in the half space $\Sigma^{\ominus}$. We will first show the central limit theorem for $\mathcal{R}(\tilde{\eta}_t)$, the process initially sampled from the invariant measure from the right edge $\tilde{\mu}$, and after extend the results to the case whenever $\eta_0$ is defined in the half space $\Sigma^{\ominus}$

We will apply a classical central limit theorem for $\alpha$-mixing sequences first found by Ibragimov and Davydov in \cite{ibragimov1975note} and \cite{davydov1970invariance} based on $\alpha$-mixing sequences of random variables. We now give a definition for the $\alpha$-mixing or strong mixing condition. Let $\{ X_i \}_{i \in \mathbb{N}}$ be a sequence of random variables. For $k \geq 1$ let $\mathcal{G}_k = \sigma(\{X_i\}_{i=1}^k)$ and $\mathcal{G}^k = \sigma(\{X_i\}_{i=k}^\infty)$.
\begin{definition}
     We say the sequence $\{ X_i \}_{i=1}^\infty$ is \textbf{$\alpha$-mixing} with mixing coefficient $g(n)$ if the sequence satisfies for each $n \geq 1$,
 \begin{equation} 
 \sup_{k} \left\{  \left| \mathbb{P}\left(  A \cap B \right) - \mathbb{P}(A)\mathbb{P}(B) \right| \mid A \in \mathcal{G}_k, \ B \in \mathcal{G}^{k+n} \right\} = g(n).
 \end{equation}
\end{definition}
\begin{theorem}[Ibragimov, Davydov]
    \label{functionalclt}
    Suppose that $\{X_i \}_{i \in \mathbb{N}}$ is a stationary and mean $0$ $\alpha$-mixing sequence of random variables with mixing coefficient $g(n)$. Suppose there exists a constant $a >0$ so that,
    \begin{equation}
        \mathbb{E}\left( X_1^{2 + a} \right) < \infty, \text{ and } \sum_{n=0}^\infty g(n)^{\frac{2}{2(2+a)}} < \infty.
    \end{equation}
    Then the process,
    \begin{equation}
        \hat{S}_n(t) = \frac{1}{\sqrt{n}}\sum_{i=1}^{\lfloor nt \rfloor} X_i + \frac{1}{\sqrt{n}}\left( nt - \lfloor nt \rfloor \right)X_{\lceil nt \rceil}
    \end{equation}
    satisfies as $n \rightarrow \infty$,
    \begin{equation}
        \{\hat{S}_n(t)\}_{t \geq 0} \Rightarrow \{W(t)\}_{t \geq 0}.
    \end{equation}
    Where $W_t$ is Brownian motion with diffusion coefficient $\sigma^2 \geq 0$, and $\Rightarrow$ denotes convergence in distribution.
\end{theorem}
In order to form the central limit theorem, we will first apply Theorem \ref{functionalclt} to the increments of $\mathcal{R}(\tilde{\eta}_t)$ at integer time points. Since $\tilde{\eta}_0$ is sampled via $\tilde{\mu}$, we then will apply Lemma \ref{largerdeviationincrement} to extend the convergence from integer to all positive real time points. We will also apply the measurable partition argument used in \cite{galves1987edge} and \cite{durrett1989contact} to show the drift coefficient $\sigma^2$ in the Brownian motion approximation is positive. 
\subsection{Showing Increments of $\mathcal{R}(\eta_t)$ are $\alpha$-Mixing}
We will first bound the $\alpha$-mixing rate of the increments of $\mathcal{R}(\tilde{\eta}_t)$ when $\tilde{\eta_0}$ is sampled via $\tilde{\mu}$. For $n \in \mathbb{N}$ we will let, $\Delta_n = \mathcal{R}(\tilde{\eta}_n) - \mathcal{R}(\tilde{\eta}_{n-1})$. We will now bound the mixing rate for $\{ \Delta_n \}_{n \in \mathbb{N}}$ using a similar approach to Galves \& Presutti in \cite{galves1987edge}.
\begin{lemma}
    \label{showalphamixing}
    Let $\tilde{\eta}_t$ be the right edge modified contact process with $\tilde{\eta}_0$ sampled via $\tilde{\mu}$. Then the collection of variables $\{ \Delta_n \}_{n \in \mathbb{N}}$ is $\alpha$ mixing with mixing coefficient,
    \begin{equation}
        g(n) \leq c_{17} \exp\left( - c_{12} n^{\frac{\delta}{8}} \right),
    \end{equation}
    for the universal constants $c_{17}$ and $c_{12}$ defined in (\ref{tgreaterthann}).
    \begin{proof}
        For $n \in \mathbb{N}$ let $\mathcal{G}_n = \sigma\left( \{ \Delta_i \}_{i=1}^n \right)$, and $\mathcal{G}^n = \sigma\left( \{ \Delta_i \}_{i=n}^\infty \right)$. Since $\tilde{\eta}_t$ is stationary, it suffices to consider times $t < 0$, and bounding the dependencies between $\mathcal{G}^n$ and $\mathcal{G}^{-1}_m = \sigma \left( \{ \Delta_i \}_{i=-m}^{-1} \right)$.

        We will define the random variables $T, I, \{\tau_i\}_{i=1}^I, \{\tau^{\emptyset,i}\}_{i=1}^I, \{ \{\eta^i_s\}_{s \geq 0 } \}_{i=1}^I$ as in section 4.3. We see that on the event $\{T < n-1\}$ that we can write using attractiveness from the right edge for $t \geq n-1$
        \begin{equation}
        \label{attractivenessmixing1}
        \mathcal{R}(\tilde{\eta}_t) = \mathcal{R}(\tilde{\eta}_T) + \mathcal{R}(\eta^T_{t-T}
        ),
        \end{equation}
        Where $\mathcal{R}(\eta^I_t)$ has the law of the right edge modified contact process with the origin initially infected conditioned on the event $\tau^\emptyset = \infty$. Since the law of $T$ is independent of $\tilde{\eta}_0$, we have by (\ref{attractivenessmixing1}) on the event $\{ T < n-1 \}$ we can conclude $\Delta_i$ is a measurable function of $\{ \mathcal{R}(\eta^T_s )\}_{s \geq 0 }$ and independent of $\mathcal{G}_{-m}^{-1}$ for all $i \geq n$. We have by (\ref{tgreaterthann}),
        \begin{equation}
            \mathbb{P}(T < n-1) \leq c_{11} \exp\left( - c_{12} (n-1)^{\frac{\delta}{8}} \right).
        \end{equation}
        And hence if we choose any $A \in \mathcal{G}_{-m}^{-1}$ and $B \in \mathcal{G}^n$ using independence between $\{ \Delta_i \}_{i=n}^\infty$ to $\mathcal{G}_{-m}^{-1}$ when $\{ T < n-1 \}$ we have,
        \begin{equation}
        \begin{aligned}
            & \left|\mathbb{P}(A \cap B |) - \mathbb{P}(A)\mathbb{P}(B) \right| \\
            & \leq \mathbb{P}(T < n-1) \\
            & \leq c_{11} \exp\left( - c_{12} (n-1)^{\frac{\delta}{8}} \right).
        \end{aligned}
        \end{equation}
        Using that $\mathcal{R}(\tilde{\eta}_t)$ has stationary increments and $m \in \mathbb{N}$ was chosen arbitrarily we can conclude,
        \begin{equation}
            g(n) \leq c_{11} \exp\left(-c_{12} (n-1)^{\frac{\delta}{8}} \right).
        \end{equation}
        Which then completes the lemma.
    \end{proof}
    \subsection{Central Limit Theorem}
\end{lemma}
We are now ready to prove the following lemma,
\begin{lemma}
    \label{cltbaselemma}    
    Let $\{ \eta_t \}_{t \geq 0}$ be the right edge modified contact process with $\eta_0$ sampled via $\tilde{\mu}$. Then as $n \rightarrow \infty$ we have,
    \begin{equation}
        \{\frac{1}{\sqrt{n}}\left( \mathcal{R}(\eta_{n t}) - \alpha n t\right) \}_{t \geq 0} \Rightarrow \{ W_t \}_{t \geq 0}.
    \end{equation}
    Where $W_t$ is Brownian motion with a drift coefficient $\sigma^2 \geq 0$.
    \begin{proof}
        Based on Theorem \ref{functionalclt}, Lemma \ref{largerdeviationincrement}, and Lemma \ref{showalphamixing} we can conclude that the process,
        \begin{equation}
            \hat{G}_n(t) = \frac{1}{\sqrt{n}}\left( \mathcal{R}(\tilde{\eta}_{\lfloor nt \rfloor}) - \alpha \lfloor nt \rfloor \right) + \frac{1}{\sqrt{n}}(nt - \lfloor nt \rfloor)\left( \mathcal{R}(\tilde{\eta} _{\lceil nt \rceil} ) - \mathcal{R}(\tilde{\eta}_{\lfloor nt \rfloor}) - \alpha  )\right),
        \end{equation}
        converges weakly to Brownian motion $W_t$ with a drift coefficient $\sigma^2 \geq 0$ as $n \rightarrow \infty$. Let,
        \begin{equation}
            G_n(t) = \frac{1}{\sqrt{n}}\left( \mathcal{R}(\tilde{\eta}_{nt}) - \alpha n t \right).
        \end{equation}
        We define the event,
        \begin{equation}
            A_n = \left\{ \sup_{i \in \mathbb{Z}, 0 \leq i \leq n^2} \sup_{0 \leq s \leq 1} \left|\mathcal{R}(\eta_{i+s}) - \mathcal{R}(\eta_i) \right| \leq n^{1/4} \right\},
        \end{equation}
        and note that on the event $A_n$,
        \begin{equation}
            \label{gncloseonan}
            \sup_{0 \leq t \leq n} \left| G_n(t) -\hat{G}_n(t)\right| \leq \frac{2}{n^{1/4}}.
        \end{equation}
        Since $\tilde{\mu}$ is a stationary measure, we can apply Lemma \ref{largerdeviationincrement} and a union bound to obtain,
        \begin{equation}
            \label{anlikely}
            \mathbb{P}(A_n^c) \leq n^2 c_{14} \exp\left(-c_{12}\left( \frac{n^{1/4}-4( \lambda_c +  \varepsilon)}{4(\lambda_c + \varepsilon)} \right)^{\frac{\delta}{8}} \right).
        \end{equation}
        Hence by (\ref{gncloseonan}) and (\ref{anlikely}), we can also conclude that $\{G_n(t)\}_{t \geq 0}$ converges weakly to $\{W_t\}_{t \geq 0}$ as $n \rightarrow \infty$, which completes the lemma.
    \end{proof}
\end{lemma}
\subsection{Drift Coefficient}
We will now show that the drift coefficient $\sigma^2$ in the Brownian motion approximation is positive. We will give an identical argument to the ones used in \cite{galves1987edge} and \cite{durrett1989contact}. We will show the following lemma.
\begin{lemma}
    \label{lemmadriftcoefficient}
    Let $\eta_t$ be the right edge modified contact process with infection rates $\lambda_i = \lambda_c$, $\lambda_e = \lambda_c + \varepsilon$, where $\varepsilon > 0$. Then we have,
    \begin{equation}
        \inf_{\eta_0 \in \Sigma^{\ominus}} \liminf_t \frac{\text{Var} \left( \mathcal{R}(\eta_t) \right)}{t} > 0
    \end{equation}
\end{lemma}
\begin{proof}
Let $t > 0$ be a constant and let $\{ \eta_s \}_{s \geq 0}$ denote the right edge modified contact process with infection rates $\lambda_i =  \lambda_c$ and $\lambda_e = \lambda_c + \varepsilon$, so that $\eta_0 \in \Sigma^{\ominus}$. We define the events for $n \in \{0,1,2, \ldots \}$,
\begin{equation}
\label{driftevents}
\begin{aligned}
    A_n &= \{ \mathcal{R}(\tilde{\eta}_t) \text{ is non-decreasing in } [n,n+1] \} \\
    B_n &= \{ \{ \eta^n_s\}_{s \geq 0} \text{ survives} \}. 
\end{aligned}
\end{equation}
We will define the sets of times, $R_0 = -1$, and for $i \geq 1$,
\begin{equation}
    R_i = \inf \{ m \in \mathbb{Z}, \ m \geq R_{i-1} +1 \mid A_m  \cap B_{m+1} \text{ occurs} \}.
\end{equation}
We also define the random variables,
\begin{equation}
    \begin{aligned}
        S_i &= \mathcal{R}(\eta_{R_{i+1}}) - \mathcal{R}(\eta_{R_{i} + 1}) \text{ for } i \geq 0 \\
        K_i &= \mathcal{R}(\eta_{R_{i}+1}) - \mathcal{R}(\eta_{R_i}) \text{ for } i \geq 1 \\
        N &= \sup \{ i \in \mathbb{Z} \mid R_i \leq t - 1\}.
    \end{aligned}
\end{equation}
For notation let $R_{N+1} = t$, and $S_{N} = \mathcal{R}(\eta_t) - \mathcal{R}(\eta_{N+1})$. We can write,
\begin{equation}
    \mathcal{R}(\eta_t) = \sum_{i=0}^N S_i + \sum_{i=1}^N K_i
\end{equation}
For any $s > 0$, let $\tau(s)$ denote the first hitting time of $\{\eta^s_{s'}\}_{s' \geq 0}$ to the all susceptible state $\emptyset$. Let $r_{0} = -1$. For $1 \leq i \leq n$ let $r_i$ and $k_i$ be constants, and let for $0 \leq j \leq n$ let $s_j$ be constants. We define the events,
\begin{equation}
    \begin{aligned}
        C_{n,t} &= A_n \cap \{ \tau(n+1) \geq t -n  -1 \} \\
    D_{u,v} &= \cap_{n=u}^{v-1} C_{n,v}^c \\
        F_i &= D_{r_i+1,r_{i+1}} \cap \{ \tau(r_i + 1) \geq r_{i+1}\text{, } \mathcal{R}(\eta^{r_i + 1}_{r_{i+1}-r_{i} - 1}) = s_i \} \\
        G_i &= A_{r_i} \cap \{ \mathcal{R}(\eta^{r_i}_1) = k_i \}. 
    \end{aligned}
\end{equation}
We now define a countable partition $\Pi$ of our outcome space $\Omega$ so that two outcomes $\omega$ and $\omega'$ are in the same partition if and only if,
\begin{enumerate}
    \item $N(\omega) = N(\omega')$
    \item $R_i(\omega) = R_i(\omega')$ for all $0 \leq i \leq N$
    \item $S_i(\omega) = S_i(\omega')$ for all $0 \leq i \leq N$
\end{enumerate}
Using the same argument as in \cite{durrett1989contact} we can write,
\begin{equation}
    \label{pipartitionprobability}
    \begin{aligned}
        & \mathbb{P}\left( N = n, \  R_i = r_i \text{, } K_i = k_i \text{ for } 1 \leq i \leq n, \ S_j = s_j \text{ for } 0 \leq j \leq n \right) \\
        &= \mathbb{P} \left( F_0 \left(  \cap_{i=1}^n F_i \cap G_i \right) \cap B_{r_{n}+1} \right).
    \end{aligned}
\end{equation}
Note the collection of events $\{ F_i \}_{i=0}^{n-1}$ and $\{ G_i \}_{i = 1}^n$ are each mutually independent, and the sigma field $\sigma(\{F_i \}_{i=1}^{n-1}, \{ G_i\}_{i=1}^n )$ is independent of the sigma field $\sigma(F_n, B_{r_n + 1})$. Using independence and (\ref{pipartitionprobability}) we have,
\begin{equation}
    \begin{aligned}
        & \mathbb{P}\left( K_i = k_i \text{ for } 1 \leq i \leq n \mid N = n, \text{ } R_i = r_i \text{ for } 1 \leq i \leq n, \text{ } S_j = s_j \text{ for } 0 \leq j \leq n \right) \\
        &= \frac{\mathbb{P}\left( F_0 \cap \left( \cap_{i=1}^n F_i \cap G_i  \right) \cap B_{r_{n}+1} \right)}{\mathbb{P}\left( F_0 \cap \left( \cap_{i=1}^n F_i \cap A_{r_i}  \right) \cap B_{r_{n}+1} \right)} \\
        &= \prod_{i=1}^n \frac{\mathbb{P}\left( G_i \cap A_{r_i} \right)}{\mathbb{P}(A_{r_i})}
    \end{aligned}
\end{equation}
Thus conditioned on $\Pi$ the collection of variables $\{ K_i \}_{i=1}^N$ are iid with the law of $\mathcal{R}(\eta_1)$ conditioned on the event $A_0$. Hence by the same argument as in \cite{galves1987edge} and \cite{durrett1989contact} we have,
\begin{equation}
\label{useconditionalindependence}
    \begin{aligned}
        \text{Var}\left( \mathcal{R}(\eta_t) \right) &\geq \mathbb{E}\left( \text{Var}\left(\mathcal{R}(\eta_t) \mid \Pi \right) \right) \\
        &= \mathbb{E}\left( \sum_{i=1}^N \text{Var}\left( K_i \mid \Pi \right) \right) \\
        &= \mathbb{E}\left( c N \right).
    \end{aligned}
\end{equation}
Where $c > 0$ is a constant independent of $t$. 

We will now show for all there exists a universal constant $c' > 0$ independent of $t$ so that for any $i \geq 1$,
\begin{equation}
    \mathbb{E}(R_{i+1} - R_i) = c' < \infty.
\end{equation}
It follows that the difference $R_{i+1} - R_i$ is a measurable function of $\{ \eta^{R_i}_s \}_{s \geq 0}$, where $\{ \eta^{R_i}_s \}_{s \geq 0}$ has the law of the right edge modified contact process with an infection at the origin conditioned on the event of survival. 

For now we continue supposing that $\{ \eta_s \}_{s \geq 0}$ is a copy of the right edge modified contact process so that $\eta_0 \in \Sigma^{\ominus}$. We note that for all $s \geq 0$ that $\mathbb{P}(A_s \cap B_{s+1} ) = e^{-1} q'' > 0$. Let $W_0 = 0$, $U_0 = 0$, and for all $m \in \mathbb{N}$,
\begin{equation}
\begin{aligned}
U_m &= \sum_{i=0}^{m-1} W_i \\
    E_m &= \{ \mathcal{R}(\eta^{U_m}_s) \text{ decreases in } [0,1]\} \\
    X_m &= \left\lceil \inf \{ s \geq 0 \mid \mathcal{R}(\eta^{U_m}_s) \text{ decreases} \} 
\right\rceil \\
Y_m &= \left\lceil 1 + \tau(U_m + 1) \right\rceil \\
W_m &= \mathbbm{1}\{ E_m\} X_m + \mathbbm{1}\{ E_m^c \} Y_m \\
Q &= \inf\{ m \in \mathbb{N} \mid W_m = \infty \}.
\end{aligned}
\end{equation}
Here we run a total of $Q$ trials until we see the renewal event that begins the next epoch $R_{i+1}$. Here $W_m$ indicates the total time required for the renewal event to fail for trial $m$, rounded up to the next highest integer, while $U_m$ represents the total time elapsed when renewal event for trial $m$ begins. 

We note $Q$ follows the Geometric$(e^{-1} q'')$ distribution. Using (\ref{survivalfunctionbound}) one can show there exists universal constants $c_{18}$ and $c_{19} > 0$ so that for all $m \in \mathbb{N}$ and $n > 0$,
\begin{equation}
\label{wmtailbound}
    \mathbb{P}\left( W_m > n \right) \leq c_{18} \exp\left(-c_{19} n^{\frac{\delta}{4}} \right)
\end{equation}
Using that $Q \sim \text{Geometric}(e^{-1}q'')$, (\ref{wmtailbound}) and the same reasoning as Lemma \ref{mixingsetuplemma}, one can show there exist universal constants $c_{20}$ and $c_{21} > 0$ so that for all $n > 0$
\begin{equation}
    \label{tailsumwm}
    \mathbb{P}\left( \sum_{m=1}^Q W_m > n \right) \leq c_{20} \exp \left(-c_{21} n^{\frac{\delta}{8}} \right).
\end{equation}
Using now that $R_{i+1} - R_i$ is a measurable function of $\{ \eta^{R_i}_s \}_{s \geq 0 }$, which has the law of the right edge modified contact process with an infection at the origin conditioned on the event $D$ of survival. Using this and (\ref{tailsumwm}) we can write for any $n > 0$,
\begin{equation}
\label{rnl1}
    \begin{aligned}
        & \mathbb{P}\left( R_{i+1} - R_i > n \right) \\
        &= \mathbb{P}\left( \sum_{m=1}^Q W_m > n \mid D \right) \\
        & \leq (q'')^{-1} \mathbb{P}\left(\sum_{m=1}^Q W_m \geq n \right) \\
        & \leq (q'')^{-1} c_{20} \exp \left( -c_{21} n^{\frac{\delta}{8}}\right)
    \end{aligned}
\end{equation}
Based on (\ref{rnl1}), we can conclude that $R_{i+1}-R_i$ has a uniformly bounded mean for all $i \geq 1$. An identical argument can be used to show $R_1$ also has a finite mean.

Define the constant,
\begin{equation}
    \label{definerho}
    \rho = \sup \left\{ \mathbb{E}\left( R_i - R_{i-1} \right) \mid i \in \mathbb{N} \right\} < \infty .
\end{equation}
Using Markov's inequality we then have when $t > 0$ is sufficiently large,
\begin{equation}
    \label{markovinequalityn}
    \begin{aligned}
    & \mathbb{P}\left( N \leq \left\lfloor \frac{t}{4 \rho} \right\rfloor \right) \\
    &= \mathbb{P}\left( \sum_{i=1}^{\lfloor \frac{t}{4 \rho} \rfloor} R_i - R_{i-1} \geq t - 1 \right) \leq \frac{1}{2}.
    \end{aligned}
\end{equation}
Using (\ref{markovinequalityn}) we have when $t > 0$ is sufficiently large,
\begin{equation}
    \label{nmeanbound}
    \mathbb{E}\left(N \right) \geq \left\lfloor \frac{t}{8 \rho} \right\rfloor.
\end{equation}
Thus combining (\ref{useconditionalindependence}) and (\ref{markovinequalityn}) we have completed the Lemma.
\end{proof}
\subsection{Proof of Theorem 2}
We now provide a proof of Theorem \ref{mainthmbrownianmotion}.
\begin{proof}[Proof of Theorem \ref{mainthmbrownianmotion}]
    Let $\{\eta_t \}_{t \geq 0}$ be the right edge modified contact process with infection rates $\lambda_i = \lambda_c$, $\lambda_e = \lambda_c + \varepsilon$ for $\varepsilon > 0$, and $\eta_0 \in \Sigma^{\ominus}$. Let the random variable $T$ be defined as in section 4.3. When $T \leq t$ we can write using attractiveness from the right edge,
    \begin{equation}
        \mathcal{R}(\eta_t) = \mathcal{R}(\eta_T) + \mathcal{R}(\eta^T_{t - T}).
    \end{equation}
    Since the variable $T$ is independent of $\eta_0 \in \Sigma^{\ominus}$, and $T < \infty$ almost surely, this allows us to deduce that the law of the increments of the right edge from $\eta_t$ and $\tilde{\eta}_t$, the process sampled from the invariant measure $\tilde{\mu}$ are tail equivalent. Further, using Lemma \ref{cltbaselemma} and \ref{lemmadriftcoefficient} that the drift coefficient $\sigma^2$ of the Brownian motion approximation for $\mathcal{R}(\tilde{\eta}_t)$ is strictly positive. This completes the theorem.
\end{proof}

\bibliographystyle{plain}
\bibliography{refs}

\begin{thebibliography}{10}

\bibitem{andjel2023contact}
Enrique Andjel and Leonardo~T Rolla.
\newblock On the contact process with modified border.
\newblock {\em arXiv preprint arXiv:2312.02059}, 2023.

\bibitem{bezuidenhout1990critical}
Carol Bezuidenhout and Geoffrey Grimmett.
\newblock The critical contact process dies out.
\newblock {\em The Annals of Probability}, 18(4):1462--1482, 1990.

\bibitem{davydov1970invariance}
Yu~A Davydov.
\newblock The invariance principle for stationary processes.
\newblock {\em Theory of Probability \& Its Applications}, 15(3):487--498, 1970.

\bibitem{duminil2018box}
Hugo Duminil-Copin, Vincent Tassion, and Augusto Teixeira.
\newblock The box-crossing property for critical two-dimensional oriented percolation.
\newblock {\em Probability Theory and Related Fields}, 171:685--708, 2018.

\bibitem{durrett1988lecture}
Richard Durrett.
\newblock Lecture notes on particle systems and percolation.
\newblock {\em Brooks/Cole Pub. Co.}, 1988.

\bibitem{durrett1983supercritical}
Richard Durrett and David Griffeath.
\newblock Supercritical contact processes on z.
\newblock {\em The Annals of Probability}, pages 1--15, 1983.

\bibitem{durrett1989contact}
Richard Durrett, Roberto~H Schonmann, and Nelson~I Tanaka.
\newblock The contact process on a finite set. iii: the critical case.
\newblock {\em The Annals of Probability}, pages 1303--1321, 1989.

\bibitem{durrett2000boundary}
Rick Durrett and Rinaldo~B Schinazi.
\newblock Boundary modified contact processes.
\newblock {\em Journal of Theoretical Probability}, 13:575--594, 2000.

\bibitem{galves1987edge}
Antonio Galves and Errico Presutti.
\newblock Edge fluctuations for the one dimensional supercritical contact process.
\newblock {\em The Annals of Probability}, 15(3):1131--1145, 1987.

\bibitem{ibragimov1975note}
Il'dar~Abdullovich Ibragimov.
\newblock A note on the central limit theorems for dependent random variables.
\newblock {\em Theory of Probability \& Its Applications}, 20(1):135--141, 1975.

\bibitem{liggett1985interacting}
Thomas~Milton Liggett and Thomas~M Liggett.
\newblock {\em Interacting particle systems}, volume~2.
\newblock Springer, 1985.

\bibitem{terra2024dynamic}
C{\'e}lio Terra.
\newblock Dynamic phenomena in interacting particle systems: Phase transitions and equilibrium.
\newblock {\em arXiv preprint arXiv:2412.16601}, 2024.

\end{thebibliography}
\end{document}